\documentclass[12pt,draftcls,onecolumn]{IEEEtran}
\IEEEoverridecommandlockouts
\usepackage{relsize}
\usepackage{amsmath,bbm}
\usepackage{amsfonts}
\usepackage{amsmath, amsthm, amssymb}
\usepackage{tikz}
\usepackage{dsfont}
\usepackage{subfig}
\usepackage{relsize}
\usetikzlibrary{arrows,automata}
\usepackage{algorithm}
\usepackage[noend]{algpseudocode}
\usepackage{comment}
\usepackage{graphicx}
\usepackage{subfig}
\usepackage[capitalise]{cleveref}

\crefname{section}{section}{sections}
\crefname{subsection}{subsection}{subsections}
\Crefname{section}{Section}{Sections}
\Crefname{subsection}{Subsection}{Subsections}

\Crefname{figure}{Figure}{Figures}

\crefformat{equation}{\textup{#2(#1)#3}}
\crefrangeformat{equation}{\textup{#3(#1)#4--#5(#2)#6}}
\crefmultiformat{equation}{\textup{#2(#1)#3}}{ and \textup{#2(#1)#3}}
{, \textup{#2(#1)#3}}{, and \textup{#2(#1)#3}}
\crefrangemultiformat{equation}{\textup{#3(#1)#4--#5(#2)#6}}%
{ and \textup{#3(#1)#4--#5(#2)#6}}{, \textup{#3(#1)#4--#5(#2)#6}}{, and \textup{#3(#1)#4--#5(#2)#6}}

\Crefformat{equation}{#2Equation~\textup{(#1)}#3}
\Crefrangeformat{equation}{Equations~\textup{#3(#1)#4--#5(#2)#6}}
\Crefmultiformat{equation}{Equations~\textup{#2(#1)#3}}{ and \textup{#2(#1)#3}}
{, \textup{#2(#1)#3}}{, and \textup{#2(#1)#3}}
\Crefrangemultiformat{equation}{Equations~\textup{#3(#1)#4--#5(#2)#6}}%
{ and \textup{#3(#1)#4--#5(#2)#6}}{, \textup{#3(#1)#4--#5(#2)#6}}{, and \textup{#3(#1)#4--#5(#2)#6}}

\newtheorem{remark}{Remark}
\newtheorem{assumption}{Assumption}
\newcommand{\be}{\begin{equation}}
\newcommand{\ee}{\end{equation}}

\usepackage{setspace}

\def\ew#1{{{\color{black}#1}}}
\def\fm#1{{{\color{black}#1}}}

\def\sm#1{{{\color{black}#1}}}

\def\ff#1{{{\color{black}#1}}}
\def\red#1{{{\color{black}#1}}}
\usepackage{mathtools}

\newtheorem{theorem}{Theorem}[section]
\newtheorem{lemma}[theorem]{Lemma}

\newcommand{\norm}[1]{\left|\left|#1\right|\right|}

\newlength{\smpagewidth}
\newlength{\smpageheight}

\title{FlexPD: A Flexible Framework of First-Order Primal-Dual Algorithms for Distributed Optimization$^*$
\thanks{$^*$ This work was supported by DARPA Lagrange HR-001117S0039} }
\author{Fatemeh Mansoori$^\dag$\thanks{$^\dag$Department of Electrical and Computer Engineering, Northwestern University, Email: ermin.wei@northwestern.edu} and Ermin Wei$^\dag$}
\date{}
\begin{document}
\maketitle
\begin{abstract}
 In this paper, we study the problem of minimizing a sum of convex objective functions, which are locally available to agents in a network. Distributed optimization algorithms make it possible for the agents to cooperatively solve the problem through local computations and communications with neighbors. Lagrangian-based distributed optimization algorithms have received significant attention in recent years, due to their exact convergence property.  However, many of these algorithms have slow convergence or are expensive to execute. In this paper, we develop a flexible framework of first-order primal-dual algorithms (FlexPD), which allows for multiple primal steps per iteration. \red{ This framework includes three algorithms, FlexPD-F, FlexPD-G, and FlexPD-C that can be used for various applications with different computation and communication limitations.} For strongly convex and Lipschitz gradient objective functions, we establish linear convergence of our proposed framework to the optimal solution. Simulation results confirm the superior performance of our framework compared to the existing methods. 
\end{abstract}
\section{Introduction}
In this paper we focus on solving the optimization problem
\be\min_{\tilde{x}\in\mathbb{R}} \sum_{i=1}^n f_i(\tilde{x})\label{mainproblem}\ee
over a network of $n$ agents (processors), which are connected with an undirected static graph $\mathcal{G(V,E)}$, where $\mathcal{V}$ and $\mathcal{E}$ denote the set of vertices and edges respectively. \footnote{\red{For representation simplicity, we consider the case with $\tilde{x}\in\mathbb{R}$. The analysis can be easily generalized to the multidimensional case. We will point out how to adapt our algorithm for $\tilde x$ in $\mathbb{R}^p$.}} Each agent $i$ in the network has access to a real-valued local objective function $f_i$, which is determined by its local data, and can only communicate with its neighbors defined by the graph. The problems of the above form arise in a wide range of applications such as sensor networks, robotics, wireless systems, and most recently in federated learning \cite{predd2007distributed, ren2007information, kekatos2013distributed, tsianos2012consensus, konevcny2016federated}. In these applications the datasets are either too large to be processed on a single processor or are collected in a distributed manner. Therefore, distributed optimization is essential to limit the data transmission, enable parallel processing, and enhance the privacy.  
A common technique in solving problem \cref{mainproblem} in a distributed way is to define local copies of the decision variable to agents. Each agent then works towards minimizing its local objective function while trying to make its decision variable equal to those of its neighbors. This can be formalized as    \be\label{eq:consensus} \min_{x}\sum_{i=1}^nf_i(x_i)\quad\mbox{s.t.}\quad x_i=x_j \quad\forall \quad (i,j)\in\mathcal{E},\ee
with $x=[x_1, x_2, ..., x_n]^\prime\in\mathbb{R}^n$. This problem is known as the \textit{consensus} optimization problem. 
\subsection{Related Work} Pioneered by the seminal works in \cite{bertsekas1989parallel,tsitsiklis1984problems}, a plethora of distributed optimization algorithms has been developed to solve problem \cref{eq:consensus}. One main category of distributed optimization algorithms to solve problem \cref{eq:consensus} is based on primal first-order (sub)gradient descent method. In particular, the authors in \cite{nedic2009distributed} presented a first-order primal iterative method, known as distributed (sub)-gradient descent (DGD), in which agents update their local estimate of the solution through a combination of a local gradient descent step and a consensus step (weighted average with neighbors variables). \red{Subsequent works have studied variants of this method with acceleration, in stochastic and asynchronous settings, and in networks with time varying graphs \cite{ram2010distributed, nedic2011asynchronous, nedic2015distributed, nedic2016stochastic, jakovetic2014fast, nesterov1983method}.}
A common property of these DGD-based algorithms along with the dual averaging algorithm in \cite{duchi2012dual} is that they can only converge to a neighborhood of the exact solution with a fixed stepsize. In order to converge to the exact solution, these algorithms need to use a diminishing stepsize, which results in a slow rate of convergence. 

On the other hand, Method of Multipliers (MM) enjoys exact convergence with constant stepsize. The method of multipliers involves a primal step, which optimizes an augmented Lagrangian function formed by adding a quadratic penalty term to the Lagrangian function, and a dual step, which takes a dual gradient ascent step \cite{hestenes1969multiplier, bertsekas2014constrained}. However, this algorithm might be extremely costly and inefficient since it requires the exact minimization at each iteration. As a remedy, \cite{jakovetic2015linear, uribe2018dual} proposed distributed MM methods with inexact minimization and provide convergence guarantees. A closely related family of methods are those based on Alternating Direction Method of Multipliers (ADMM) \cite{gabay1983chapter,eckstein2015understanding}, which also have exact convergence and enjoy  good numerical performance, \cite{boyd2011distributed, wei2012distributed, mota2013d, shi2014linear, iutzeler2016explicit, wei20131}. Instead of one minimization of the primal variables per iteration as in MM, ADMM partitions the primal variables into two sets and takes two minimization steps (one for each subset) per iteration. Versions of distributed ADMM with inexact minimization and exact convergence were developed in \cite{ling2015dlm, mokhtari2015decentralized, liu2018acceleration, mokhtari2016decentralized, yu2019communication}.

Recently, distributed gradient based methods with gradient tracking technique have been developed \cite{shi2015extra, qu2017harnessing, sun2019convergence, nedic2017achieving}. In addition to a consensus step on the iterates like in DGD, these methods also takes weighted average on the gradients. These methods are shown to converge with constant stepsize to the exact optimal solution linearly if the objective function is strongly convex and has Lipschitz gradient. Although these algorithms do not involve dual variables explicitly, they can be viewed as Arrow-Hurwicz-Uzawa primal-dual gradient methods \cite{arrow1958studies} for augmented Lagrangian, which replace the primal minimization problem in MM with a single primal gradient descent step. The gradient step can be carried out locally by an agent using one gradient evaluation and communication with neighbors. The algorithms in \cite{shi2015proximal,wu2018decentralized, mokhtari2016dsa, zeng2015extrapush, xi2017dextra} are proximal, asynchronous, and stochastic versions of these primal-dual algorithm for directed and undirected graphs.  


For all the distributed algorithms with exact convergence guarantees, they either suffer from computational complexity caused by the (approximate) minimization involved at each iteration, or have slower numerical performance.  We can view MM and Arrow-Hurwicz-Uzawa as two extremes, where the primal minimization step with respect to the augmented Lagrangian is either solved exactly or with only one gradient step, and aim to bridge the gap in between. Motivated by this observation, we propose a flexible family of primal-dual methods (FlexPD) that can take arbitrarily many ($T$) primal gradient steps of the augmented Lagrangian before taking a dual step and provide exact linear convergence guarantees for constant stepsize. This family of methods gives flexibility and controls over the trade-off between the complexity and performance. Our proposed method 
shares some similarities to a methods in \cite{uribe2018dual, liu2018acceleration}. Work \cite{uribe2018dual} aims to quantify the number of steps required to reach $\varepsilon$-neighborhood of the optimal solution instead of exact convergence. 
In \cite{liu2018acceleration}, the authors studied a different problem of the form $\min_x f(x)+g(Ax)$ and showed that a primal-dual hybrid gradient (PDHG) method, which combines a minimization in the primal space and a fixed number of proximal gradient steps in the dual space at each iteration, can guarantee global convergence. 

In addition to introducing the flexibility in choosing number of primal steps, our method also offers flexibility in reducing the amount of communication and/or gradient computation operations, which are two of the building blocks of distributed optimization methods. We develop two variants of FlexPD, where for the $T$ primal updates, outdated information regarding neighbors' values or local gradient is used instead of the current one. These variants also enjoy linear rate of convergence. To this end, our paper is related to the growing literature on communication aware distributed methods, including methods designed to limit communication by graph manipulation techniques~\cite{chow2016expander}, special communication protocols~\cite{lan_communication-efficient_2017, sahu_communication-efficient_2018,shen_towards_2018}, algorithmic design~\cite{federated5,shamir2014communication,zhang2012communication,zhang2015disco} or quantization/encoding schemes~\cite{alistarh2017qsgd,rabbat_quantized_2005,berahas2019nested} and methods that aim to balance communication and computation loads depending on application-specific requirements~\cite{jaggi_comm_eff_coord_ascent,tsianos2012communication,berahas2018balancing,nokleby_distributed_2017}. 


\subsection{Contributions} 
\ff{ We develop a \textbf{Flex}ible \textbf{P}rimal-\textbf{D}ual framework (FlexPD)  to bridge the gap between the method of multipliers and primal-dual methods. In the version called FlexPD-F algorithm, we replace the primal space minimization step in MM with a finite number ($T$) of \textbf{F}ull gradient descent steps. Each primal gradient descent step in FlexPD-F involves one gradient evaluation and one round of communication. Hence each iteration of the method consists of $T$ gradient and communication operations for $T$ primal updates followed by a dual gradient update. To address the scenarios where communication or computation is costly, we further develop  FlexPD-G and FlexPD-C variants, which utilizes outdated information. Each iteration of FlexPD-G involves $T$ \textbf{G}radient evaluations and one communication for $T$ primal updates and a dual gradient update, whereas FlexPD-C has $T$ rounds of \textbf{C}ommunication, one gradient evaluation for $T$ primal steps and one dual step per iteration. The framework is based a general form of augmented Lagrangian, which is flexible in the augmentation term. For our proposed framework, we establish the linear convergence of all three algorithms to the exact solution with constant stepsize  for strongly convex objective function with Lipschitz gradient. Also, due to our general form of the augmented Lagrangian, the algorithms presented in \cite{shi2015extra, nedic2017achieving, qu2017harnessing} can be considered as special cases of our general framework for particular choices of the augmentation term and with $T=1$.  Our numerical experiments demonstrate the improved performance of our algorithms over those with one gradient descent step in the primal space.} \red{Part of the results for FlexPD-C has appeared in our earlier work \cite{mansoori2019general}, this paper proposes two additional novel  methods, FlexPD-F, FlexPD-G. The three methods combined form a flexible framework for distributed first-order primal-dual methods. }

 The rest of this paper is organized as follows: \cref{sec:alg} describes the development of our general framework,  \cref{sec:conv} contains the convergence analysis,  \cref{sec:simul} presents the numerical experiments, and  \cref{sec:conclusions} contains the concluding remarks.

 \textbf{Notations:} A vector is viewed as a column vector. For a matrix $A$, we write $A_{ij}$ to denote the component of $i^{th}$ row and $j^{th}$ column. Notations $I$ and $\textbf{0}$ are used for the identity and zero matrix. \sm{We denote the largest and second smallest eigenvalues of a symmetric matrix $A$, by $\rho(A)$ and $s(A)$ respectively. Also, for a symmetric matrix $A$, $aI\preceq A\preceq bI$ means that the eigenvalues of $A$ lie in $[a,b]$ interval. For two symmetric matrices $A$ and $B$ we use $A\preceq B$ if and only if $B-A$ is positive semi-definite.} 
For a vector $x$, $x_i$
denotes the $i^{th}$ component of the vector.
We use $x'$ and $A'$ to
denote the transpose of a vector $x$ and a matrix $A$ respectively.
 We use standard Euclidean norm (i.e., 2-norm) unless otherwise noted, i.e., for a vector $x$ in $\mathbb{R}^n$, $\norm{x}=\left(\sum_{i=1}^n x_i^2\right)^{\frac{1}{2}}$. \ff{We use the weighted norm notation $\norm{x}_A$ to represent $x^\prime Ax$ for any positive semi-definite $A$.} \footnote{\red{In \cref{sec:conv}, we disregard the positive semi-definiteness requirement of the weight matrix when using this notation. Ultimately, all the weight matrices are shown to be positive definite.}} For a real-valued function $f:\mathbb{R}\rightarrow \mathbb{R}$, the gradient vector of $f$ at $x$  is denoted by $\nabla f(x)$ and the Hessian matrix is denoted by $\nabla^2f(x)$.
\section{Algorithm Development} \label{sec:alg} In this section, we derive the flexible framework of primal-dual algorithms that allows for multiple primal steps at each iteration. 
To develop our algorithm, we rewrite problem \cref{eq:consensus} in the following compact form 
\be\min_x f(x)\quad\mbox{s.t}\quad Ax=0, \label{eq:Newconsensus}\ee where $f(x)=\sum_{i=1}^n f_i(x_i)$ with $x=[x_1, x_2, ..., x_n]^\prime\in\mathbb{R}^n$, and $Ax=0$ represents all equality constraints. One example of the matrix  $A$ is the edge-node incidence matrix of the network graph \cite{bertsimas1997introduction}, i.e., $A\in\mathbb{R}^{\epsilon\times n}$, $\epsilon=\vert \mathcal{E}\vert$, whose null space is spanned by the vector of all ones.
\fm{Row $l$ of matrix $A$ corresponds to edge $l$, connecting vertices $i$ and $j$, and has $+1$ in column $i$ and $-1$ in column $j$ (or vice versa) and $0$ in all other columns.} 
We denote by $x^*=[\tilde{x}^*, \tilde{x}^*,...,\tilde{x}^*]'$ the minimizer of problem \cref{eq:Newconsensus}. To achieve exact convergence, we develop our framework based on the Lagrange multiplier methods. We form the following augmented Lagrangian \be L_a(x,\lambda)=f(x)+\lambda^\prime Ax+\frac{1}{2}x^\prime Bx, \label{auglag}\ee  
where $\lambda\in\mathbb{R}^\epsilon$ is the vector of Lagrange multipliers. Each dual variable $\lambda_l$ is associated with an edge $l=(i,j)$ and thus coupled between two agents and is updated by one of them. The set of dual variables that agent $i$ updates is denoted by $\Lambda_i$. 

We adopt the following assumptions on our problem.
\begin{assumption} \label{funcprop}
The local objective functions $f_i(x)$ are $m-$ strongly convex, twice differentiable, and $L-$ Lipschitz gradient, i.e., \[mI\preceq \nabla^2f_i(x)\preceq LI.\]
\end{assumption}
\begin{assumption}\label{Bprop}
Matrix $B\in\mathbb{R}^{n\times n}$ is a symmetric positive semi-definite matrix, has the same null space as matrix $A$, i.e., $Bx=0$ only if $Ax=0$, and is compatible with network topology, i.e., $B_{ij}\neq 0$ only if $(i,j)\in\mathcal{E}$.
\end{assumption}
We assume these conditions hold for the rest of the paper. The first assumption on the eigenvalues of the Hessian matrix of local objective functions is a standard assumption in proving the global linear rate of convergence. The second assumption requires matrix $B$ to represent the network topology, which is required for distributed implementation and the other assumptions on matrix $B$ are needed for convergence guarantees. With $A$ being the edge-node incidence matrix, one examples of matrix $B$ is $B=c A^\prime A$, with $A^\prime A$ being the graph Laplacian matrix. Another example is the weighted Laplacian matrix. 
We develop our algorithm based on the following form of primal-dual iteration. 
\begin{equation}\begin{aligned}x^{k+1}=x^{k}-\alpha\nabla_xL_a(x^k,\lambda^k)= x^k-\alpha\nabla f(x^k)-\alpha A^\prime\lambda^k-\alpha B x^k,\label{Primalupdate}\end{aligned}\end{equation}
\[\lambda^{k+1}=\lambda^{k}+\beta\nabla_\lambda L_a(x^{k+1},\lambda^k)=\lambda^k+\beta Ax^{k+1},\]
\fm{where $\alpha$ and $\beta$ are constant stepsize parameters.}
\begin{algorithm}
\caption{FlexPD-F}
\label{alg:FlexPD-F}
\begin{algorithmic}
\State{Initialization: for $i=1, 2, ..., n$ each agent $i$ picks $x_i^0$, sets $\lambda_{l_i}^0=0\quad\forall{\lambda_{l_i}\in\Lambda_i}$, and determines $\alpha$, $\beta$, and $T<\infty$}
\For{$k=1,2,...$}
 \State{\[x_i^{k+1,0}=x_i^k\]}
\For{$t=1,2,...,T$}
\State{\be x_i^{k+1,t}=x_i^{k+1,t-1}-\alpha\nabla f_i(x_i^{k+1,t-1})-\alpha\sum_{l=1}^e A^\prime_{il}\lambda_l^k-\alpha\sum_{j=1}^n B_{ij}x_j^{k+1,t-1} \label{eq:primalF}\ee}
\EndFor
\textbf{end for}
\State {\[x_i^{k+1}=x_i^{k+1,T}, \quad \lambda_{l_i}^{k+1}=\lambda_{l_i}^{k}+\beta\sum_{j=1}^n A_{l_ij}x_j^{k+1}\quad\forall{\lambda_{l_i}\in\Lambda_i}.\]}
\EndFor
\textbf{end for}
\end{algorithmic}
\end{algorithm}
In contrary to MM and ADMM algorithms that update the primal variable by minimizing the augmented Lagrangian, this iteration uses one gradient descent step to update the primal variable, and therefore is less expensive to execute. Different variations of the above iteration have been used to solve the consensus optimization problem \cref{eq:Newconsensus} \cite{shi2015extra,nedic2017achieving}, however, the convergence of MM is shown to be faster \cite{mokhtari2016decentralized}. This observation motivates the development of a framework that controls the trade-off between the performance and the execution complexity of primal-dual algorithms. In our FlexPD-F algorithm, the primal variable is updated through $T$ full gradient descent steps at each iteration. The intuition behind this method is that by increasing the number of gradient descent steps from 1 to $T$ at each iteration the resulting primal variable is closer to the minimizer of augmented Lagrangian at that iteration, due to the strong convexity of the augmented Lagrangian [c.f. \cref{funcprop} and \cref{Bprop}].  

\ff{\begin{remark}[General applicability of proposed algorithm]
We note that our proposed algorithm can be applied to more general settings. When the decision variable $x$ is in \red{$\mathbb{R}^p$}, we can apply our algorithm by using the Kronecker product of $A$ and p-dimensional identity matrix instead of $A$. Iterations \eqref{eq:primalF}-\eqref{eq:C} would be implemented for each of the $p$ components. The algorithm can also be adopted to other choices for matrix $A$ -- such as weighted edge-node incidence matrix \cite{wu2018decentralized}, graph Laplacian matrix \cite{tutunov2019distributed}, and weighted Laplacian matrix \cite{mokhtari2016decentralized, berahas2018balancing} -- and corresponding matrix $B$. Lastly, although our framework is motivated by a distributed setting, our proposed algorithms can be implemented for general equality constrained minimization problems of form \cref{eq:Newconsensus} in both centralized and distributed settings.
\end{remark}}
\begin{algorithm}
\caption{FlexPD-G}
\label{alg:FlexPD-G}
\begin{algorithmic}
\State{Initialization: for $i=1, 2, ..., n$ each agent $i$ picks $x_i^0$, sets $\lambda_{l_i}^0=0\quad\forall{\lambda_{l_i}\in\Lambda_i}$, and determines $\alpha$, $\beta$, and $T<\infty$}
\For{$k=1,2,...$}
\vspace{-.5cm}
 \State{\[x_i^{k+1,0}=x_i^k\]}
\For{$t=1,2,...,T$}
\State{\be x_i^{k+1,t}=x_i^{k+1,t-1}-\alpha\nabla f_i(x_i^{k+1,t-1})-\alpha\sum_{l=1}^e A^\prime_{il}\lambda_l^k-\alpha\sum_{j=1}^n B_{ij}x_j^{k} \label{eq:primalG}\ee}
\EndFor
\textbf{end for}
\State {\[x_i^{k+1}=x_i^{k+1,T},\quad \lambda_{l_i}^{k+1}=\lambda_{l_i}^{k}+\beta\sum_{j=1}^n A_{l_ij}x_j^{k+1}\quad\forall{\lambda_{l_i}\in\Lambda_i}.\]}
\EndFor
\textbf{end for}
\end{algorithmic}
\end{algorithm}
 \begin{algorithm}
\caption{FlexPD-C}
\label{alg:FlexPD-C}
\begin{algorithmic}
\State{Initialization: for $i=1, 2, ..., n$ each agent $i$ picks $x_i^0$, sets $\lambda_{l_i}^0=0\quad\forall{\lambda_{l_i}\in\Lambda_i}$, and determines $\alpha$, $\beta$, and $T<\infty$}
\For{$k=1,2,...$}
 \State{\[x_i^{k+1,0}=x_i^k\]}
\For{$t=1,2,...,T$}
\State{\begin{align}\label{eq:C} x_i^{k+1,t}=x_i^{k+1,t-1}-\alpha\nabla f_i(x_i^{k})-\alpha\sum_{l=1}^e A^\prime_{il}\lambda_l^k-\alpha\sum_{j=1}^n B_{ij}x_j^{k+1,t-1}\end{align}}
\EndFor
\textbf{end for}
\State {\[x_i^{k+1}=x_i^{k+1,T}, \quad \lambda_{l_i}^{k+1}=\lambda_{l_i}^{k}+\beta\sum_{j=1}^n A_{l_ij}x_j^{k+1}\quad\forall{\lambda_{l_i}\in\Lambda_i}.\]}
\EndFor
\textbf{end for}
\end{algorithmic}
\end{algorithm}
We next verify that FlexPD-F algorithm can be implemented in a distributed way. We note that at each outer iteration $k+1$ of \cref{alg:FlexPD-F}, each agent updates its primal variable by taking $T$ gradient descent steps. At each inner iteration $t$, each agent $i$ has access to its local gradient $\nabla f_i(x_i^{k+1,t-1})$ and the primal and dual variables of its neighbors, $\lambda_j^k$ and $x_j^{k+1,t-1}$, through communication, and computes $x_i^{k+1,t}$ using Eq. \cref{eq:primalF}. Agent $i$ then communicates this new variable to its neighbors. After $T$ gradient descent steps, agent $i$ updates its associated dual variables by using $x_i^{k+1,T}$ and $x_j^{k+1,T}$ from its neighbors. We note that each iteration of this algorithm involves $T$ gradient evaluation and $T$ rounds of communication for each agent. For settings with communication or computation limitations, the FlexPD-F algorithm might not be efficient. In what follows, we develop two other classes of algorithms which are adaptive to such settings. 
\par \sm{To keep communication limited, we introduce the FlexPD-G algorithm, in which the agents communicate once per iteration.} In our proposed algorithm in \cref{alg:FlexPD-G}, at each iteration $k+1$, agent $i$ goes through $T$ inner iterations. At each inner iteration $t$, each agent $i$ reevaluates its local gradient and updates its primal variable by using Eq. \cref{eq:primalG}. After $T$ inner iterations agent $i$ communicates its primal variable $x_i^{k+1,T}$ with its neighbors and updates its corresponding dual variables $\lambda_{l_i}$ by using local $x_i^{k+1,T}$ and $x_j^{k+1,T}$ from its neighbors. \ff{We note that for this algorithm to converge, we need an extra assumption on matrix $B$, which is introduced in \cref{ass:Bbound}.}
\par \sm{Finally, to avoid computational complexity, we develop the FlexPD-C algorithm, in which the gradient is evaluated once per iteration and is used for all primal updates in that iteration.} In our proposed framework in \cref{alg:FlexPD-C}, at each iteration $k+1$, agent $i$ computes its local gradient $\nabla f_i(x_i^k)$, and performs a predetermined number ($T$) of primal updates by \ew{repeatedly communicating with neighbors without recomputing its gradient [c.f. Eq \cref{eq:C}]}. Each agent $i$ then updates its corresponding dual variables $\lambda_{l_i}$ by using local $x_i^{k+1,T}$ and $x_j^{k+1,T}$ from its neighbors.
\par \fm{Under \Cref{funcprop}, there exists a unique optimal solution $\tilde{x}^*$ for problem \cref{mainproblem} and thus a unique $x^*$ \ew{exists}, at which the function value is bounded. Moreover, since $Null(A)\neq\emptyset$, the Slater's condition is satisfied. Consequently, strong duality holds and a dual optimal solution $\lambda^*$ exists.} We note that the projection of $\lambda^k$ in the null space of matrix $A^\prime $ would not affect the performance of algorithm, and therefore, the optimal dual solution is not uniquely defined, since for any optimal dual solution $\lambda^*$ the dual solution $\lambda^*+u$, where $u$ is in the null space of $A^\prime $, is also optimal. If the algorithm starts at $\lambda=0$, then all the iterates $\lambda^k$ are in the column space of $A$ and hence orthogonal to null space of $A^\prime $. Without loss of generality, we assume that in all three algorithms  $\lambda^0=0$, and when we refer to an optimal dual solution $\lambda^*$, we assume its projection onto the null space of $A^\prime $ is $0$. We note that $(x^*,\lambda^*)$ is a fixed point of FlexPD-F, FlexPD-G, and FlexPD-C iterations.
\section {Convergence Analysis} \label{sec:conv} 
In this section, we analyze the convergence properties of the three algorithms presented in the previous section. In \cref{GDproof}, \cref{Gproof}, and \cref{Cproof} we establish the linear rate of convergence for FlexPD-F, FlexPD-G, and FlexPD-C algorithms respectively. To start the analysis, we note that the dual update for all three algorithms has the following form 
\be \lambda^{k+1}=\lambda^k+\beta Ax^{k+1}.\label{eq:lambda} \ee
We also note that the KKT condition for problem \cref{eq:Newconsensus} implies \be\nabla f(x^*)+A^\prime\lambda^*=0,\quad Ax^*=0, \quad\mbox{and}\quad Bx^*=0, \label{kkt}\ee
where the last equality comes from the fact that $Null(B)=Null(A)$. Before diving in the convergence analysis of the algorithms, we state and prove an important inequality which is a useful tool in establishing the desired properties. 
\begin{lemma}\label{lemma:NormSquare}
	For any vectors $a$, $b$, and scalar $\xi>1$, we have
	\[(a+b)^\prime  (a+b) \leq \frac{\xi}{\xi-1}a^\prime a + \xi b^\prime  b.\]
\end{lemma}
\begin{proof}
Since $\xi>1$, we have $\frac{\xi-1}{\xi}+\frac{1}{\xi}=1$ and we can write the right hand side as
\begin{align*}
 \frac{\xi}{\xi-1}a^\prime a + \xi b^\prime  b= (\frac{\xi}{\xi-1}a^\prime a + \xi b^\prime  b)(\frac{\xi-1}{\xi}+\frac{1}{\xi})=a^\prime a+b^\prime  b+\frac{1}{\xi-1}a^\prime a+(\xi-1)b^\prime  b.
\end{align*}

We also have that 
\[0\leq \left(\sqrt{\frac{1}{\xi-1}}a-\sqrt{\xi-1}b\right)^\prime  \left(\sqrt{\frac{1}{\xi-1}}a-\sqrt{\xi-1}b\right) =  \frac{1}{\xi-1}a^\prime a+(\xi-1)b^\prime  b - 2a^\prime b,\] which implies that 
$\frac{1}{\xi-1}a^\prime a+(\xi-1)b^\prime  b\geq 2a^\prime b .$

We can then combine this into the previous equality and obtain the result.
\end{proof}
\subsection{Convergence Analysis of FlexPD-F}\label{GDproof}
In order to analyze the convergence properties of FlexPD-F algorithm, we first rewrite the primal update in \cref{alg:FlexPD-F} in the following compact form  
\be 
x^{k+1,t}=U x^{k+1,t-1}-\alpha \nabla f(x^{k+1,t-1})-\alpha A^\prime \lambda^k, \label{eq:GD}
\ee
 where $U=I-\alpha B$. We next proceed to prove the linear convergence rate for our proposed algorithm. In \cref{lemma:gradF-F}, \cref{lemma:strcvx-F}, and \cref{lemma:toCompare-F} we prove some key relations that we use to establish an upper bound on the Lyapunov function in \cref{thm:delta-F}. We then combine this bound with the result of \cref{lemma:MBound-F} to establish the linear rate of convergence for the FlexPD-F algorithm in \cref{thm:linConv-F}.
\begin{lemma}\label{lemma:gradF-F} Consider the primal-dual iterates as in \cref{alg:FlexPD-F}, we have
	\begin{align*}&\alpha (\nabla f(x^{k+1,T-1})- \nabla f(x^*) )= U(x^{k+1,T-1}-x^{k+1,T}) + \\& \alpha  (\beta  A^\prime A-B) (x^{k+1,T} - x^*) -\alpha  A^\prime (\lambda^{k+1} -\lambda ^*).\end{align*} 
\end{lemma}
\begin{proof}
	Consider the primal update in Eq. \cref{eq:GD} at iteration $k+1$ with $t=T$, we have
\[\alpha \nabla f(x^{k+1,T-1})  = Ux^{k+1,T-1}-x^{k+1,T} -\alpha  A^\prime  \lambda^k.\]
	Moreover,  we can substitute the expression for $\lambda^k$ from Eq. \cref{eq:lambda} into the previous equation and have
	\begin{align*} &\alpha \nabla f(x^{k+1,T-1})= Ux^{k+1,T-1}-x^{k+1,T} -\alpha  A^\prime  (\lambda ^{k+1} - \beta  A x^{k+1,T}) \\&= U(x^{k+1,T-1}-x^{k+1,T}) + \alpha(  \beta A^\prime A-B) x^{k+1,T} -\alpha  A^\prime \lambda^{k+1}.\end{align*} 
By using Eq. \cref{kkt}, we have
	$\alpha\nabla f(x^*)=-A^\prime \lambda^*-( \beta  A^\prime A-B)x^*$, which we subtract from the above equation to obtain the result.
\end{proof}
\begin{lemma}\label{lemma:MBound-F} If $U\succ \textbf{0}$ , i.e., $\alpha<\frac{1}{\rho(B)}$ for $B\neq\textbf{0}$, we have
\begin{align*}
\norm{x^{k+1,T-1}-x^*}^2_{U}
	+ \frac{\alpha }{\beta }\norm{\lambda^k-\lambda^*}^2\leq \Gamma_F^{T-1}\big(\norm{x^{k}-x^*}^2_{U}
	+ \frac{\alpha }{\beta }\norm{\lambda^k-\lambda^*}^2\big),
\end{align*}
with $\Gamma_F=\max\Bigg\lbrace{1+\frac{p\alpha\beta\rho(AA^\prime)}{p-1},p\Big(\sqrt{\rho(U)}+\alpha L\sqrt{\rho(U^{-1})} \Big)^2 \Bigg\rbrace}$ for any $p>1$.
\end{lemma}
\begin{proof} Consider the primal update in Eq. \cref{eq:GD} at $t=T-1$, by subtracting $x^*$ from both sides of this equality we have \[x^{k+1,T-1}-x^*=U x^{k+1,T-2}-x^*-\alpha\nabla f(x^{k+1,T-2})-\alpha A^\prime \lambda^k.\]
By using Eq. \cref{kkt} we have
$0=\alpha\big(\nabla f(x^*)+A^\prime\lambda^*+Bx^*\big),$
which we add to the previous equality to obtain
\[x^{k+1,T-1}-x^*=U(x^{k+1,T-2}-x^*)-\alpha\nabla \big(f(x^{k+1,T-2})-\nabla f(x^*)\big)-\alpha A^\prime (\lambda^k-\lambda^*).\]
Hence,  \[\norm{x^{k+1,T-1}-x^*}^2=
\norm{U(x^{k+1,T-2}-x^*)-\alpha\nabla \big(f(x^{k+1,T-2})-\nabla f(x^*)\big)-\alpha A^\prime (\lambda^k-\lambda^*)}^2.\]
By using the result of \cref{lemma:NormSquare}, and Lipschitz continuity of $\nabla f$, for any $p, q>1$ we have 
\begin{align*}&\norm{x^{k+1,T-1}-x^*}^2\leq\frac{p\alpha^2}{p-1}(\lambda^k-\lambda^*)^\prime AA^\prime(\lambda^k-\lambda^*)+ p\Big\Vert U(x^{k+1,T-2}-x^*)\\&-\alpha\nabla \big(f(x^{k+1,T-2})-\nabla f(x^*)\big)\Big\Vert^2\leq\frac{p\alpha^2}{p-1}(\lambda^k-\lambda^*)^\prime AA^\prime(\lambda^k-\lambda^*)+p\Big(q (x^{k+1,T-2}\\&-x^*)^\prime U^2(x^{k+1,T-2}-x^*)+\frac{q\alpha^2 L^2}{q-1}(x^{k+1,T-2}-x^*)^\prime(x^{k+1,T-2}-x^*)\Big).
\end{align*}
We now consider the second term in the right hand side of the previous inequality, by using the fact that $U\succ \textbf{0}$ we have
\begin{align*}&q (x^{k+1,T-2}-x^*)^\prime U^2(x^{k+1,T-2}-x^*)+\frac{q\alpha^2 L^2}{q-1}\norm{x^{k+1,T-2}-x^*}^2=(x^{k+1,T-2}\\&-x^*)^\prime U^{1/2}\Big(qU+\frac{q\alpha^2 L^2}{q-1}U^{-1}\Big) U^{1/2}(x^{k+1,T-2}-x^*).\end{align*}
By using the above two inequalities, we obtain
\be\begin{aligned}\label{eq:NormBound-F}&\norm{x^{k+1,T-1}-x^*}^2\leq\frac{p\alpha^2\rho(AA^\prime)}{p-1}\norm{\lambda^k-\lambda^*}^2\\&+p\Big(q\rho(U)+\frac{q\alpha^2 L^2}{q-1}\rho(U^{-1})\Big)\norm{x^{k+1,T-2}-x^*}^2_{U}.\end{aligned}\ee
Since the above inequality holds for all $q>1$, we can find the parameter $q$ that makes the right hand smallest, which would give us the most freedom to choose algorithm parameters.  The term $q\rho(U)+\frac{q\alpha^2L^2}{q-1} \rho(U^{-1})$ is convex in $q$ and to minimize it we set the derivative to 0 and have 
$q = 1+ 
\alpha L\frac{\sqrt{\rho(U^{-1}) }}{\sqrt{\rho(U)}}.$ Therefore, $q\rho(U)+\frac{q\alpha^2L^2}{q-1} \rho(U^{-1})=\Big(\sqrt{\rho(U)}+\alpha L\sqrt{\rho(U^{-1})}\Big)^2.$
We also note that matrix $B$ is positive semi-definite, therefore,
$\norm{x^{k+1,T-1}-x^*}^2_{U} \leq\norm{x^{k+1,T-1}-x^*}^2.$
By using the previous two relations and Eq. \cref{eq:NormBound-F} we obtain
\begin{align*}&\norm{x^{k+1,T-1}-x^*}^2_{U}\leq\frac{p\alpha^2\rho(AA^\prime)}{p-1}\norm{\lambda^k-\lambda^*}^2\\&+p\Big(\sqrt{\rho(U)}+\alpha L\sqrt{\rho(U^{-1})}\Big)^2\norm{x^{k+1,T-2}-x^*}^2_{U}.\end{align*}
By adding $\frac{\alpha}{\beta}\norm{\lambda^k-\lambda^*}^2$ to both sides of the previous inequality, we have
\begin{align*} &\norm{x^{k+1,T-1}-x^*}^2_{U}+\frac{\alpha}{\beta}\norm{\lambda^k-\lambda^*}^2\leq\Big(1+\frac{p\alpha\beta\rho(AA^\prime)}{p-1}\Big)\frac{\alpha}{\beta}\norm{\lambda^k-\lambda^*}^2\\&+p\Big(\sqrt{\rho(U)}+\alpha L\sqrt{\rho(U^{-1})}\Big)^2\norm{x^{k+1,T-2}-x^*}^2_{U}.\end{align*}
We can now write the previous inequality as follows
\be\begin{aligned}\label{eq:MBound-F}&\norm{x^{k+1,T-1}-x^*}^2_{U}+\frac{\alpha}{\beta}\norm{\lambda^k-\lambda^*}^2\leq\Gamma_F\Big(\norm{x^{k+1,T-2}-x^*}^2_{U}+\frac{\alpha}{\beta}\norm{\lambda^k-\lambda^*}^2\Big),\end{aligned}\ee
with $\Gamma_F=\max\Bigg\lbrace{1+\frac{p\alpha\beta\rho(AA^\prime)}{p-1},p\Big(\sqrt{\rho(U)}+\alpha L\sqrt{\rho(U^{-1})}\Big)^2 \Bigg\rbrace}.$

By applying inequality \cref{eq:MBound-F} recursively we complete the proof.
\end{proof}

\begin{lemma}\label{lemma:strcvx-F} Consider the primal-dual iterates as in \cref{alg:FlexPD-F}, if $\alpha<\frac{1}{\rho(B)}$, we have
 \begin{align*}  &\norm{x^{k+1,T}-x^*}^2_{P_F} +\frac{\alpha }{\beta }\norm{\lambda^{k+1}-\lambda^k}^2+ \norm{x^{k+1,T}-x^{k+1,T-1}}^2  _{Q_F} \leq\\&
\norm{x^{k+1,T-1}-x^*}^2_{U}- \norm{x^{k+1,T}-x^*}^2  _{U}
 + \frac{\alpha }{\beta }\left(\norm{\lambda^k-\lambda^*}^2 - \norm{\lambda^{k+1}-\lambda^*}^2\right),
 \end{align*}
 with $P_F=2\alpha  m I-\alpha \eta_1 I +2\alpha B-\alpha\beta  A^\prime A$ and $Q_F=U- \frac{\alpha  L^2}{\eta_1}I$. 
 \end{lemma}
\begin{proof}
	From strong convexity of function $f(x)$, we have
	\begin{align*}&2\alpha  m \norm{x^{k+1,T}-x^*}^2 \leq 2\alpha  (x^{k+1,T}-x^*)^\prime  \big(\nabla f(x^{k+1,T}) - \nabla f(x^*)\big) = 2\alpha  (x^{k+1,T}-x^*)^\prime \\& \big(\nabla f(x^{k+1,T})-\nabla f(x^{k+1,T-1})\big) + 2\alpha  (x^{k+1,T}-x^*)^\prime  \big(\nabla f(x^{k+1,T-1})-\nabla f(x^*)\big),\end{align*} where we added and subtracted a term $(x^{k+1,T}-x^*)^\prime  \nabla f(x^{k+1,T-1})$.
	We can substitute the equivalent expression of $ \alpha (\nabla f(x^{k+1,T-1})- \nabla f(x^*) )$ from \cref{lemma:gradF-F} and have
		\be\begin{aligned}\label{eq:update1-F}&2\alpha  m \norm{x^{k+1,T}-x^*}^2\leq 2\alpha (x^{k+1,T}-x^*)^\prime  \big(\nabla f(x^{k+1,T})-\nabla f(x^{k+1,T-1})\big)+\\&2\alpha  (x^{k+1,T}-x^*)^\prime  (\beta  A^\prime A - B)(x^{k+1,T}-x^*) + 2(x^{k+1,T}-x^*)^\prime  U(x^{k+1,T-1}-x^{k+1,T}) \\&-2\alpha (x^{k+1,T}-x^*)^\prime   A^\prime (\lambda^{k+1} -\lambda ^*). \end{aligned}\ee
		We also have, by Young's inequality, for all $\eta_1>0$, 
		\begin{align*}&2\alpha(x^{k+1,T}-x^*)^\prime  \big(\nabla f(x^{k+1,T})-\nabla f(x^{k+1,T-1})\big) \leq\frac{\alpha }{\eta_1}\Vert\nabla f(x^{k+1,T})-\nabla f(x^{k+1,T-1})\\&\Vert^2+\alpha \eta_1 \norm{x^{k+1,T}-x^*}^2\leq\alpha \eta_1 \norm{x^{k+1,T}-x^*}^2+\frac{\alpha  L^2 }{\eta_1}\norm{x^{k+1,T}-x^{k+1,T-1}}^2,\end{align*}
		where the second inequality holds by the Lipschitz continuity of $\nabla f(.)$.
By the dual update Eq. \cref{eq:lambda} and feasibility of $x^*$, we have 
$Ax^{k+1} = \frac{1}{\beta }(\lambda^{k+1}-\lambda^k),\quad Ax^* = 0.$
These two equations combined yields
$\alpha (x^{k+1}-x^*)^\prime  A^\prime (\lambda^{k+1} -\lambda ^*) = \frac{\alpha }{\beta }(\lambda^{k+1}-\lambda^k)^\prime (\lambda^{k+1}-\lambda^*).$ Hence we can rewrite Eq. \cref{eq:update1-F} as 
\begin{align*}&2\alpha  m \norm{x^{k+1,T}-x^*}^2\leq 2\alpha (x^{k+1,T}-x^*)^\prime (\beta A^\prime A-B)(x^{k+1,T}-x^*)+\\& \alpha \eta_1 \norm{x^{k+1,T}-x^*}^2+\frac{\alpha  L^2}{\eta_1}\norm{x^{k+1,T}-x^{k+1,T-1}}^2 - 2(x^{k+1,T}-x^*)^\prime U \\& (x^{k+1,T}-x^{k+1,T-1}) -2\frac{\alpha }{\beta }(\lambda^{k+1}-\lambda^k)^\prime  (\lambda^{k+1}-\lambda^*).
 \end{align*}
 We focus on the last two terms. First, since matrix $U$ is symmetric positive definite, we have 
\begin{align*}&-2(x^{k+1,T}-x^*)^\prime U(x^{k+1,T}-x^{k+1,T-1})=\\& \norm{x^{k+1,T-1}-x^*}^2_{U} - \norm{x^{k+1,T}-x^*}^2_{U}- \norm{x^{k+1,T}-x^{k+1,T-1}}^2_{U}.\end{align*}
 Similarly, 
\[-2\frac{\alpha }{\beta }(\lambda^{k+1}-\lambda^k)^\prime  (\lambda^{k+1}-\lambda^*) = \frac{\alpha }{\beta }\Big(\norm{\lambda^k-\lambda^*}^2 - \norm{\lambda^{k+1}-\lambda^*}^2-\norm{\lambda^{k+1}-\lambda^k}^2\Big).\]
 Now we combine the terms in the preceding three relations and have 
\begin{align*}&2\alpha  m \norm{x^{k+1,T}-x^*}^2\leq 2\alpha (x^{k+1,T}-x^*)^\prime  (\beta A^\prime A-B)(x^{k+1,T}-x^*)+ \alpha \eta_1 \Vert x^{k+1,T}\\&-x^*\Vert^2+\frac{\alpha  L^2}{\eta_1}\norm{x^{k+1,T}-x^{k+1,T-1}}^2 +\norm{x^{k+1,T-1}-x^*}^2_{U}- \norm{x^{k+1,T}-x^*}^2_{U}\\&- \norm{x^{k+1,T}-x^{k+1,T-1}}^2_{U} + \frac{\alpha }{\beta }\left(\norm{\lambda^k-\lambda^*}^2 - \norm{\lambda^{k+1}-\lambda^*}^2-\norm{\lambda^{k+1}-\lambda^k}^2\right).
 \end{align*}
 We now use Eq. \cref{eq:lambda} together with the fact that $Ax^*=0$ to obtain $\norm{\lambda^{k+1}-\lambda^k}^2 =\beta ^2(x^{k+1}-x^*)^\prime (A^\prime A)(x^{k+1}-x^*)$. By using this relation rearranging the terms, we complete the proof.
\end{proof}
\begin{lemma}\label{lemma:toCompare-F} Consider the primal-dual iteration as in \cref{alg:FlexPD-F}, if $\alpha<\frac{1}{\rho(B)}$, we have for any $c,d, g>1$
 \begin{align*}
  &\norm{\lambda^{k+1}-\lambda^*}^2\leq \frac{d}{\alpha ^2 s(AA^\prime )} \left(\frac{e}{e-1} \rho(U^2)+e \alpha ^2L^2\right)\norm{x^{k+1,T-1}-x^{k+1,T}}^2 \\ \nonumber&+ \frac{d}{(d-1)\alpha ^2 s(AA^\prime )} \left(\frac{g}{g-1} \alpha ^2 \rho (\beta  A^\prime A-B)^2+g\alpha ^2L^2\right)\norm{x^{k+1,T} - x^*}^2.
  \end{align*}
\end{lemma}
\begin{proof}
We recall from \cref{lemma:gradF-F} that \begin{align*}&\alpha  A^\prime (\lambda^{k+1} -\lambda ^*)= \\&
	U(x^{k+1,T-1}-x^{k+1,T}) + \alpha  (\beta  A^\prime A-B) (x^{k+1,T} - x^*) - \alpha \big(\nabla f(x^{k+1,T-1})- \nabla f(x^*)\big).\end{align*}
By adding and subtracting a term of $\nabla f(x^{k+1,T})$ and taking squared norm from both sides, we have
\begin{align*}&\norm{\alpha  A^\prime (\lambda^{k+1} -\lambda ^*)}^2 = \big\Vert U(x^{k+1,T-1}-x^{k+1,T}) + \alpha  (\beta  A^\prime A-B) (x^{k+1,T} - x^*) \\&- \alpha \left(\nabla f(x^{k+1,T-1})- \nabla f(x^{k+1,T} )\right)- \alpha \left(\nabla f(x^{k+1,T})- \nabla f(x^*) \right)\big\Vert^2.\end{align*}

We now apply \cref{lemma:NormSquare} and have for any $d,e,g>1$, $\norm{\alpha  A^\prime (\lambda^{k+1} -\lambda ^*)}^2 \leq$
\begin{align*}& d\left(\frac{e}{e-1}\norm{x^{k+1,T-1}-x^{k+1,T}}^2_{ U^2} +e \alpha ^2\norm{\nabla f(x^{k+1,T-1})- \nabla f(x^{k+1,T})}^2\right) \\&+ \frac{d}{d-1}\left(\frac{g}{g-1} \alpha ^2\norm{x^{k+1,T} - x^*}^2_{(\beta  A^\prime A-B)^2}+g\alpha ^2\norm{\nabla f(x^{k+1,T})- \nabla f(x^*) }^2\right).
\end{align*}

Since $\lambda^0=0$ and $\lambda^{k+1} = \lambda^k + \beta  Ax^{k+1,T}$, we have that $\lambda^k$ is in the column space of $A$ and hence orthogonal to the null space of $A^\prime $, hence we have
$\norm{\alpha  A^\prime (\lambda^{k+1} -\lambda ^*)}^2\geq \alpha ^2s(AA^\prime )\norm{\lambda^{k+1}-\lambda^*}^2$. By using this inequality and the Lipschitz gradient property of function $f(.)$ , we have 
 \begin{align*}
 &\alpha ^2 s(AA^\prime )\norm{\lambda^{k+1}-\lambda^*}^2\leq d\left((x^{k+1,T-1}-x^{k+1,T})^\prime  \left[\frac{e}{e-1}  U^2+e \alpha ^2L^2I\right](x^{k+1,T-1}-x^{k+1,T}) \right) \\&+ \frac{d}{d-1}\left((x^{k+1,T} - x^*)^\prime   \left[\frac{g}{g-1} \alpha ^2 (\beta  A^\prime A-B)^2+g\alpha ^2L^2I\right] (x^{k+1,T} - x^*)\right).
 \end{align*}
We next bound matrices $\frac{e}{e-1}  U^2+e \alpha ^2L^2I$ and $\frac{g}{g-1} \alpha ^2 (\beta  A^\prime A-B)^2+\alpha ^2gL^2I$ by their largest eigenvalues to obtain the result. 
\end{proof}
\begin{theorem}\label{thm:delta-F}
For $0<\eta_1<2m$, if
$\alpha< \frac{1}{L^2/\eta_1+\rho(B)}$ and $\beta<\frac{2m-\eta_1}{\rho(A^\prime A)}$, there exists $\delta_F>0$, such that 
	\begin{align*}&\norm{x^{k+1,T}-x^*}^2_{U}+ \frac{\alpha }{\beta }\norm{\lambda^{k+1}-\lambda^*}^2\leq\frac{\norm{x^{k+1,T-1}-x^*}^2_{U}
	+ \frac{\alpha }{\beta }\norm{\lambda^k-\lambda^*}^2}{1+\delta_F}. 
\end{align*}
\end{theorem}

\begin{proof}
To show the result, we will show that
\begin{align*}   &\delta_F\left(\norm{x^{k+1,T}-x^*}^2_{U}+ \frac{\alpha }{\beta }\norm{\lambda^{k+1}-\lambda^*}^2\right)\leq  \norm{x^{k+1,T-1}-x^*}^2_{U} \\&- \norm{x^{k+1,T}-x^*}^2_{U}
+ \frac{\alpha }{\beta }\left(\norm{\lambda^k-\lambda^*}^2 - \norm{\lambda^{k+1}-\lambda^*}^2\right),
\end{align*}
for some $\delta_F>0$. By comparing the above inequality to the result of \cref{lemma:strcvx-F}, it suffices to show that
\[\delta_F\left(\norm{x^{k+1,T}-x^*}^2_{U}+ \frac{\alpha }{\beta }\norm{\lambda^{k+1}-\lambda^*}^2\right)\leq  \norm{x^{k+1,T}-x^*}^2_{P_F}+\norm{x^{k+1,T}-x^{k+1,T-1}}^2_{Q_F}.\]
We collect the terms and we will focus on showing 
\begin{align*}\label{ineq:linearRateToShow-F}   &\norm{\lambda^{k+1}-\lambda^*}^2\leq  \frac{\beta }{\delta_F\alpha }\norm{x^{k+1,T}-x^{k+1,T-1}}^2_{Q_F}+\frac{\beta }{\delta_F\alpha }\norm{x^{k+1,T}-x^*}^2_{P_F-\delta_F U}.\end{align*}

By comparing this to the result of \cref{lemma:toCompare-F}, we need for some $\delta_F>0$,   \[\frac{\beta }{\delta_F\alpha }Q_F \succcurlyeq \frac{d}{\alpha ^2 s(AA^\prime )} \left(\frac{e}{e-1} \rho(U^2)+e \alpha ^2L^2\right)I\] and 
  \[\frac{\beta }{\delta_F\alpha }\big(P_F-\delta_F U\big)\succcurlyeq  \frac{d}{(d-1)\alpha ^2 s(AA^\prime )} \big(\frac{g}{g-1} \alpha ^2 \rho (\beta  A^\prime A-B)^2+g\alpha ^2L^2\big)I.\]
Since \cref{lemma:toCompare-F} holds for all $e,g>1$, we can find the parameters $e$ and $g$ to make the right hand side of the previous two relations smallest, which would give us the most freedom to choose algorithm parameters.  The term $\frac{e}{e-1} \rho(I-\alpha  B)^2 + e \alpha ^2L^2$ is convex in $e$ and to minimize it we set derivative to 0 and have 
$ e = 1+ \frac{\rho(U)}{\alpha  L}.$
Similarly, we choose $g$ to be
$g =1 + \frac{ \rho (\beta  A^\prime A-B)}{ L}$
With these parameter choices, we have
$\frac{e}{e-1} \rho(U^2)+e \alpha ^2L^2= (\rho(U)+ \alpha  L)^2,$ and
$\frac{g}{g-1} \alpha ^2 \rho (\beta  A^\prime A-B)^2+\alpha ^2gL^2 = \alpha ^2 (\rho(\beta  A^\prime A-B)+L)^2.$ By substituting these relations and the by considering the definitions of $P_F$ and $Q_F$ from \cref{lemma:strcvx-F}, the above inequalities are satisfied if
\[\frac{\beta}{\delta_F\alpha}(1-\alpha \rho(B)-\frac{\alpha L^2}{\eta_1}) \geq \frac{d}{\alpha s(AA')} (1+ \alpha L)^2\] and
  \[\frac{\beta}{\delta_F\alpha}\Big(2\alpha m-\alpha\eta_1-\alpha\beta\rho(A'A) -\delta_F\big(1-\alpha\rho(B)\big)\Big) \geq  \frac{d \alpha^2 (\rho(\beta A'A-B)+L)^2}{(d-1)\alpha s(AA')}.
\] For the first inequality, we can multiply both sides by $\delta_F$ and rearrange the terms and have
$
\delta_F \leq \frac{\beta (1-\alpha \rho(B)-\frac{\alpha L^2}{\eta_1})s(AA')  }{d(1+ \alpha L)^2}.
$
We can similarly solve for the second inequality and have
$\delta_F\leq\frac{ \beta(2\alpha m-\alpha\eta_1-\alpha\beta\rho (A'A) )}{ \left(\frac{d }{(d-1) s(AA')} \alpha^2 (\rho(\beta A'A-B)+L)^2+\beta\big(1-\alpha \rho(B)\big)\right)}.
$
This give some $\delta_F>0$ as long as $\eta_1<2m$, 
$\alpha< \frac{1}{L^2/\eta_1+\rho(B)},$ and
$\beta<\frac{2m-\eta_1}{\rho(A^\prime A)}.$ 
The parameter set is nonempty and thus we can find a $\delta_F>0$ which establishes the desired result.
\end{proof}	
\begin{theorem}\label{thm:linConv-F} Consider \cref{alg:FlexPD-F} with $T>1$, recall the definition of $\Gamma_F$ from \cref{lemma:MBound-F}, and define $ z^k=\begin{bmatrix} x^k \\ \lambda^k \end{bmatrix}, \quad \mathcal{G}_F=\begin{bmatrix} U & \textbf{0} \\ \textbf{0} & \frac{\alpha}{\beta} I \end{bmatrix}.$ For any $0<\eta_1<2m$ and  $0<\tilde{\delta}_F<\delta_F$, if $ \alpha<\min\Big\lbrace{\frac{1}{L^2/\eta_1+\rho(B)}, \frac{(1+\tilde{\delta}_F )^{\frac{1}{2(T-1)}}-1}{L+\rho(B)\big((1+\tilde{\delta}_F )^{\frac{1}{2(T-1)}}-1\big)}, \frac{(1+\tilde{\delta}_F )^{\frac{1}{T-1}}-1}{\beta \rho(AA^\prime)}\Big\rbrace}$ and $\beta<\frac{2m-\eta_1}{\rho(A^\prime A)}$ we have
\[\norm{z^{k+1}-z^*}^2_{\mathcal{G}_F}\leq\frac{\Gamma_F^{T-1}}{1+\tilde{\delta}_F }\norm{z^k-z^*}^2_{\mathcal{G}_F},\quad\mbox{with $\frac{\Gamma_F^{T-1}}{1+\tilde{\delta}_F }<1$,}\]
that is $\norm{z^k-z^*}_{\mathcal{G}_F}$ converges Q-linearly to $0$  and consequently $\norm{x^k-x^*}_U$ converges R-linearly to $0$.	
\end{theorem}
\begin{proof}
We note that with $0<\eta_1<2m$ , $\alpha<\frac{1}{L^2/\eta_1+\rho(B)}$, and $\beta<\frac{2m-\eta_1}{\rho(A^\prime A)}$ the result of \cref{thm:delta-F} holds and we have for every $\tilde{\delta}_F <\delta_F$
\begin{align*}   &\norm{x^{k+1,T}-x^*}^2_{U} + \frac{\alpha }{\beta }\norm{\lambda^{k+1}-\lambda^*}^2\leq  \frac{1}{1+\tilde{\delta}_F }\Big(\norm{x^{k+1,T-1}-x^*}^2_{U}
	\\&+ \frac{\alpha }{\beta }\norm{\lambda^k-\lambda^*}^2 \Big)\leq  \frac{\Gamma_F^{T-1}}{1+\tilde{\delta}_F }\left(\norm{x^{k}-x^*}^2_{U}
	+ \frac{\alpha }{\beta }\norm{\lambda^k-\lambda^*}^2 \right),
\end{align*}
where the second inequality is based on the result of \cref{lemma:MBound-F}.
Finally, for $T>1$ we need to show that $\frac{\Gamma_F^{T-1}}{1+\tilde{\delta}_F }<1.$
We note that from  \cref{lemma:MBound-F}, we have
\[\Gamma_F=\max\Big\lbrace{1+\frac{p\alpha\beta\rho(AA^\prime)}{p-1}, p\Big(\sqrt{\rho(U)}+\alpha L\sqrt{\rho(U^{-1})}\Big)^2 \Big\rbrace}.\]
If $\Gamma_F=1+\frac{p\alpha\beta\rho(AA^\prime)}{p-1}$, we need $\frac{\big(1+\frac{p\alpha\beta\rho(AA^\prime)}{p-1}\big)^{T-1}}{1+\tilde{\delta}_F }<1,$
which is equivalent to $1+\frac{p\alpha\beta\rho(AA^\prime)}{p-1}<(1+\tilde{\delta}_F )^{\frac{1}{T-1}}.$
Hence, we need
$\frac{p}{p-1}<\frac{(1+\tilde{\delta}_F )^{\frac{1}{T-1}}-1}{\alpha\beta\rho(AA^\prime)}.$
We also have $p>1$ and hence $\frac{p}{p-1}>1$. Therefore, we need to choose $\alpha$ such that 
the upper bound on $\frac{p}{p-1}$ is greater than one, i.e., $\alpha<\frac{(1+\tilde{\delta}_F )^{\frac{1}{T-1}}-1}{\beta\rho(AA^\prime)}.$
We next consider the case when $\Gamma_F=p\Big(\sqrt{\rho(U)}+\alpha L\sqrt{\rho(U^{-1})}\Big)^2$. In this case, we need 
$\frac{\Big(p\Big(\sqrt{\rho(U)}+\alpha L\sqrt{\rho(U^{-1})}\Big)^2\Big)^{T-1}}{1+\tilde{\delta}_F }<1,$
which is equivalent to
$p<\frac{(1+\tilde{\delta}_F )^{\frac{1}{T-1}}}{\Big(\sqrt{\rho(U)}+\alpha L\sqrt{\rho(U^{-1})}\Big)^2}.$
By considering the fact that $p>1$, we need to choose $\alpha$ such that the right hand side of the previous inequality is greater than one. i.e.,
$\Big(\sqrt{\rho(U)}+\alpha L\sqrt{\rho(U^{-1})}\Big)^2<(1+\tilde{\delta}_F )^{\frac{1}{T-1}}.$
By taking square root from both sides of the above inequality and by replacing $\rho(U)$ and $\rho(U^{-1})$ by their upper bounds, we have
$1+\alpha L\sqrt{\frac{1}{1-\alpha\rho(B)}}<(1+\tilde{\delta}_F )^{\frac{1}{2(T-1)}}.$
We note that $0<1-\alpha\rho(B)\leq 1$ and thus $\sqrt{\frac{1}{1-\alpha\rho(B)}}<\frac{1}{1-\alpha\rho(B)}$. Therefore, we need 
$1+\frac{\alpha L}{1-\alpha\rho(B)}<(1+\tilde{\delta}_F )^{\frac{1}{2(T-1)}},$
which will be satisfied if $\alpha<\frac{(1+\tilde{\delta}_F )^{\frac{1}{2(T-1)}}-1}{L+\rho(B)\big((1+\tilde{\delta}_F )^{\frac{1}{2(T-1)}}-1\big)}.$
\end{proof}
\subsection{Convergence Analysis of FlexPD-G}\label{Gproof}
In order to analyze the convergence properties of FlexPD-G algorithm, we first rewrite the primal update in \cref{alg:FlexPD-G} in the following compact form  
\be 
x^{k+1,t}=x^{k+1,t-1}-\alpha \nabla f(x^{k+1,t-1})-\alpha A^\prime \lambda^k-\alpha Bx^{k}, \label{eq:G}
\ee
\ff{In addition to \cref{funcprop} and \cref{Bprop}, we adopt the following assumption on matrix $B$ in this section.
\begin{assumption} \label{ass:Bbound} The spectral radius of matrix $B$ is upper bounded by $m$, i.e., $\rho(B)<m$. 
\end{assumption}
We note that in our proposed form of augmented Lagrangian, the choice of matrix $B$ is flexible and we can scale it by any positive number such that it satisfies \cref{ass:Bbound}.}
We next proceed to prove the linear convergence rate for our proposed algorithm. In \cref{lemma:gradF-G}, \cref{lemma:strcvx-G}, and \cref{lemma:lambdabound-G} we prove some key relations that we use to establish an upper bound on the Lyapunov function in \cref{thm:delta-G}. We then combine this bound with the result of \cref{lemma:MBound-G} to establish the linear rate of convergence for the FlexPD-G algorithm in \cref{thm:linConv-G}.

\begin{lemma}\label{lemma:gradF-G} Consider the primal-dual iteration as in \cref{alg:FlexPD-G}, we have
	\begin{align*} &\alpha \big(\nabla f(x^{k+1,T-1})- \nabla f(x^*) \big)=\\&  (x^{k+1,T-1}-x^{k+1,T}) -\alpha B(x^k-x^*) -\alpha  A^\prime (\lambda^{k+1} -\lambda ^*)+\alpha\beta AA^\prime(x^{k+1,T}-x^*).\end{align*}
\end{lemma}
\begin{proof}
	Consider the primal update in Eq. \cref{eq:G} at iteration $k+1$ with $t=T$, we have
\[\alpha \nabla f(x^{k+1,T-1})  = x^{k+1,T-1}-x^{k+1,T} -\alpha B x^k-\alpha  A^\prime  \lambda^k.\]
	Moreover, we have for dual variable $\lambda^k$,
	$\lambda ^k = \lambda ^{k+1} - \beta  A x^{k+1,T}.$ We can substitute this expression for $\lambda^k$ into the previous equation and have
	\begin{align*} &\alpha \nabla f(x^{k+1,T-1})  = x^{k+1,T-1}-x^{k+1,T} -\alpha Bx^k-\alpha  A^\prime  (\lambda ^{k+1} - \beta  A x^{k+1,T}) \\&= (x^{k+1,T-1}-x^{k+1,T})-\alpha Bx^k + \alpha\beta A^\prime Ax^{k+1,T} -\alpha  A^\prime \lambda^{k+1}.\end{align*} 
By using the optimality conditions in Eq. \cref{kkt} we have
	$\alpha\nabla f(x^*)=-\alpha A^\prime \lambda^*+\alpha \beta  A^\prime Ax^*-\alpha Bx^*.$ By subtracting the previous two relations, we complete the proof.
\end{proof}
\begin{lemma}\label{lemma:MBound-G} Consider the primal-dual iteration as in \cref{alg:FlexPD-G}, we have
\begin{align*}
&\norm{x^{k+1,T-1}-x^*}^2
	+ \frac{\alpha }{\beta }\norm{\lambda^k-\lambda^*}^2+\alpha\rho(B)\norm{x^k-x^*}^2\leq\\& \Gamma_G^{T-1}\Big(c_1\norm{x^{k} -x^*}^2 
	+ \frac{\alpha }{\beta }\norm{\lambda^k-\lambda^*}^2\Big),
\end{align*}
 with  $c_1=1+\alpha\rho(B)$ and $\Gamma_G=\max\Big\lbrace{\bar{p}(1+
 \alpha L)^2, 1+\frac{\bar{p}\bar{q}\alpha\beta\rho(AA^\prime)}{\bar{p}-1}, 1+\frac{\bar{p}\bar{q}\alpha\rho(B)}{(\bar{p}-1)(\bar{q}-1)}\Big\rbrace}$ for any $\bar{p}, \bar{q}>1$.
\end{lemma}
\begin{proof} Consider the primal update in Eq. \cref{eq:G} at $t=T-1$, by subtracting $x^*$ from both sides of this equality we have \[x^{k+1,T-1}-x^*=x^{k+1,T-2}-x^*-\alpha\nabla f(x^{k+1,T-2})-\alpha A^\prime \lambda^k-\alpha Bx^k.\]

By using the optimality condition in Eq. \cref{kkt}, we have 
$0=\alpha\big(\nabla f(x^*)+A^\prime\lambda^*+Bx^*\big).$
By adding the previous two relations, we obtain
\[x^{k+1,T-1}-x^*=x^{k+1,T-2}-x^*-\alpha\nabla \big(f(x^{k+1,T-2})-\nabla f(x^*)\big)-\alpha A^\prime (\lambda^k-\lambda^*)-\alpha B(x^k-x^*).\]
Hence, \begin{align*}&\norm{x^{k+1,T-1}-x^*}^2=\\&
\norm{x^{k+1,T-2}-x^*-\alpha\nabla \big(f(x^{k+1,T-2})-\nabla f(x^*)\big)-\alpha A^\prime (\lambda^k-\lambda^*)-\alpha B(x^k-x^*)}^2.\end{align*}
By using the result of  \cref{lemma:NormSquare}, and Lipschitz continuity property of $\nabla f$, for any $\bar{p},  \bar{q}, \bar{r}>1$ we have 
\begin{align*}&\norm{x^{k+1,T-1}-x^*}^2\leq\bar{p}\norm{x^{k+1,T-2}-x^*-\alpha\nabla \big(f(x^{k+1,T-2})-\nabla f(x^*)\big)}^2+\frac{\bar{p}}{\bar{p}-1}\\&\norm{\alpha A^\prime (\lambda^k-\lambda^*)-\alpha B(x^k-x^*)}^2\leq \bar{p}\big(\frac{\bar{r}}{\bar{r}-1}+\bar{r}\alpha^2L^2\big)\norm{x^{k+1,T-2}-x^*}^2\\&+\frac{\bar{p}}{\bar{p}-1}\Big(\bar{q}\alpha^2\rho(AA^\prime)\norm{\lambda^k-\lambda^*}^2+\frac{\bar{q}\alpha^2\rho^2(B)}{\bar{q}-1}\norm{x^k-x^*}^2\Big).\end{align*}
Since the above inequality holds for all $\bar{r}>1$, we can find the parameter $\bar{r}$ that makes the right hand side smallest and provides tightest upper bound.  The term $\frac{\bar{r}}{\bar{r}-1}+\bar{r}\alpha^2 L^2$ is convex in $\bar{r}$ and to minimize it we set derivative to 0 and have 
$ \bar{r} = 1+ \frac{1}{\alpha L}.$
Therefore, $\frac{\bar{r}}{\bar{r}-1}+\bar{r}\alpha^2 L^2=(1+\alpha L)^2.$
Hence, we have
\begin{align*}&\norm{x^{k+1,T-1}-x^*}^2\leq\\& \bar{p}(1+\alpha L)^2\norm{x^{k+1,T-2}-x^*}^2+\frac{\bar{p}\bar{q}\alpha^2\rho(AA^\prime)}{\bar{p}-1}\norm{\lambda^k-\lambda^*}^2+\frac{\bar{p}\bar{q}\alpha^2\rho^2(B)}{(\bar{q}-1)(\bar{p}-1)}\norm{x^k-x^*}^2.\end{align*}
By adding $\frac{\alpha}{\beta}\norm{\lambda^k-\lambda^*}^2$ and $\alpha\rho(B)\norm{x^k-x^*}$ to both sides of the previous inequality, we have
\begin{align*}&\norm{x^{k+1,T-1}-x^*}^2+\frac{\alpha}{\beta}\norm{\lambda^k-\lambda^*}^2+\alpha\rho(B)\norm{x^k-x^*}^2\leq \\&\bar{p}(1+\alpha L)^2\norm{x^{k+1,T-2}-x^*}^2+\big(1+\frac{\bar{p}\bar{q}\alpha\beta\rho(AA^\prime)}{\bar{p}-1}\big)\frac{\alpha}{\beta}\norm{\lambda^k-\lambda^*}^2\\&+\big(1+\frac{\bar{p}\bar{q}\alpha\rho(B)}{(\bar{q}-1)(\bar{p}-1)}\big)\alpha\rho(B)\norm{x^k-x^*}^2.\end{align*}
We can write the previous inequality as follows
\be\begin{aligned}\label{eq:newMBound-G}&\norm{x^{k+1,T-1}-x^*}^2+\frac{\alpha}{\beta}\norm{\lambda^k-\lambda^*}^2+\alpha\rho(B)\norm{x^k-x^*}^2\\&\leq \Gamma_G\Big( \norm{x^{k+1,T-2}-x^*}^2+\frac{\alpha}{\beta}\norm{\lambda^k-\lambda^*}^2+\alpha\rho(B)\norm{x^k-x^*}^2\Big),\end{aligned}\ee
with $\Gamma_G=\max\Big\lbrace{\bar{p}(1+
 \alpha L)^2, 1+\frac{\bar{p}\bar{q}\alpha\beta\rho(AA^\prime)}{\bar{p}-1}, 1+\frac{\bar{p}\bar{q}\alpha\rho(B)}{(\bar{p}-1)(\bar{q}-1)}\Big\rbrace}.$
By applying inequality \cref{eq:newMBound-G} recursively we obtain the result.
\end{proof}
\begin{lemma}\label{lemma:strcvx-G} Consider the primal-dual iteration in \cref{alg:FlexPD-G}, for any $\eta_2$, $\eta_3>0$ we have
\begin{align*} &\norm{x^{k+1,T}-x^*}^2_{P_G}+c_2\norm{x^{k+1,T}-x^{k+1,T-1}}^2
+c_3\norm{x^k-x^*}^2
\leq\Vert x^{k+1,T-1}-\\&x^*\Vert^2-c_1\norm{x^{k+1,T}-x^*}^2+\frac{\alpha}{\beta}\norm{\lambda^k-\lambda^*}^2-\frac{\alpha}{\beta}\norm{\lambda^{k+1}-\lambda^*}^2+\alpha\rho(B)\norm{x^k-x^*}^2,\end{align*}
with $P_G=2\alpha mI-\alpha(\eta_2+\eta_3)I-\alpha\beta A^\prime A-\alpha\rho(B)$, $c_1=1+\alpha\rho(B)$, $c_2=1-\frac{\alpha L^2}{\eta_2}$, and  $c_3=\alpha\rho(B)(1-\frac{\rho(B)}{\eta_3})$.
\end{lemma}

\begin{proof}
	From the strong convexity of $f(x)$, we have 
	\begin{align*}&2\alpha m\norm{x^{k+1,T}-x^*}^2\leq2\alpha\times  (x^{k+1,T}-x^*)^\prime  \big(\nabla f(x^{k+1,T}) - \nabla f(x^*)\big)=2\alpha  (x^{k+1,T}-x^*)^\prime \\&\big(\nabla f(x^{k+1,T})-\nabla f(x^{k+1,T-1})\big) +2\alpha (x^{k+1,T}-x^*)^\prime  \big(\nabla f(x^{k+1,T-1})-\nabla f(x^*)\big),\end{align*} where we add and subtract a term $(x^{k+1,T}-x^*)^\prime  \nabla f(x^{k+1,T-1})$.
	We can now substitute the equivalent expression of $ \alpha (\nabla f(x^{k+1,T-1})-\nabla f(x^*) )$ from \cref{lemma:gradF-G} and have
		\be\begin{aligned}\label{eq:update1-G}&2\alpha  m \norm{x^{k+1,T}-x^*}^2\leq 2\alpha  (x^{k+1,T}-x^*)^\prime  \big(\nabla f(x^{k+1,T})-\nabla f(x^{k+1,T-1})\big)\\&+2\alpha\beta  (x^{k+1,T}-x^*)^\prime   A^\prime A (x^{k+1,T}-x^*)+ 2(x^{k+1,T}-x^*)^\prime  (x^{k+1,T-1}-x^{k+1,T}) \\& -2\alpha (x^{k+1,T}-x^*)^\prime   A^\prime (\lambda^{k+1} -\lambda ^*)-2\alpha(x^{k+1,T}-x^*)^\prime B(x^k-x^*). \end{aligned}\ee
		We also have by Young's inequality, 
		 for all $\eta_2>0$, 
		\begin{align*}&2\alpha  (x^{k+1,T}-x^*)^\prime  \big(\nabla f(x^{k+1,T})-\nabla f(x^{k+1,T-1})\big) \leq \\&\alpha \eta_2 \norm{x^{k+1,T}-x^*}^2+\frac{\alpha }{\eta_2}\norm{\nabla f(x^{k+1,T})-\nabla f(x^{k+1,T-1})}^2.\end{align*}
By Lipschitz property of $\nabla f$ we have
 \begin{align*}&2\alpha  (x^{k+1,T}-x^*)^\prime  \big(\nabla f(x^{k+1,T})-\nabla f(x^{k+1,T-1})\big) \leq\\&\alpha \eta_2 \norm{x^{k+1,T}-x^*}^2+\frac{\alpha  L^2 }{\eta_2}\norm{x^{k+1,T}-x^{k+1,T-1}}^2.\end{align*}

Similarly, for any $\eta_3>0$, we have
$-2\alpha(x^{k+1,T}-x^*)^\prime B(x^k-x^*)=2\alpha(x^*-x^{k+1,T})^\prime B(x^k-x^*)\leq\alpha\eta_3\norm{x^{k+1,T}-x^*}^2+\frac{\alpha\rho^2(B)}{\eta_3}\norm{x^k-x^*}^2.$
By the dual update Eq. \cref{eq:lambda} and the feasibility of $x^*$, we have 
$Ax^{k+1} = \frac{1}{\beta }(\lambda^{k+1}-\lambda^k)$ and $Ax^* = 0.$
These two equations combined yields
$\alpha (x^{k+1}-x^*)^\prime  A^\prime (\lambda^{k+1} -\lambda ^*) = \frac{\alpha }{\beta }(\lambda^{k+1}-\lambda^k)^\prime (\lambda^{k+1}-\lambda^*).$ We also have  $-2(x^{k+1,T}-x^*)^\prime(x^{k+1,T}-x^{k+1,T-1})=\Vert x^{k+1,T-1}-x^*\Vert^2 -\norm{x^{k+1,T} -x^*}^2-\norm{x^{k+1,T}-x^{k+1,T-1}}^2$
and similarly $-2\frac{\alpha }{\beta }(\lambda^{k+1}-\lambda^k)^\prime  (\lambda^{k+1}-\lambda^*) = \frac{\alpha }{\beta }\Big(\norm{\lambda^k-\lambda^*}^2 - \norm{\lambda^{k+1}-\lambda^*}^2-\norm{\lambda^{k+1}-\lambda^k}^2\Big).$
Now we combine the preceding three relations and Eq. \cref{eq:update1-G} to obtain
\begin{align*}&2\alpha m\norm{x^{k+1,T}-x^*}^2\leq\alpha\eta_2\norm{x^{k+1,T}-x^*}^2+\frac{\alpha L^2}{\eta_2}\norm{x^{k+1,T}-x^{k+1,T-1}}^2\\+&2\alpha\beta(x^{k+1,T}-x^*)^\prime A^\prime A(x^{k+1,T}-x^*)+\norm{x^{k+1,T-1}-x^*}^2-\norm{x^{k+1,T}-x^*}^2\\&-\norm{x^{k+1,T}-x^{k+1,T-1}}^2
+\frac{\alpha}{\beta}\norm{\lambda^k-\lambda^*}^2-\frac{\alpha}{\beta}\norm{\lambda^{k+1}-\lambda^*}^2-\frac{\alpha}{\beta}\norm{\lambda^{k+1}-\lambda^k}^2\\&+\alpha\eta_3\norm{x^{k+1,T}-x^*}^2+\frac{\alpha\rho^2(B)}{\eta_3}\norm{x^k-x^*}^2.\end{align*}
We now use Eq. \cref{eq:lambda} together with the fact that $Ax^*=0$ to obtain \[\norm{\lambda^{k+1}-\lambda^k}^2 =\beta ^2(x^{k+1}-x^*)^\prime (A^\prime A)(x^{k+1}-x^*).\] We substitute this relation in its preceding inequality and subtract $\alpha\rho(B)\norm{x^{k+1,T}-x^*}^2$ form and add $\alpha\rho(B)\norm{x^k-x^*}^2$ to its both sides. By rearranging the terms we obtain
\begin{align*} &(x^{k+1,T}-x^*)^\prime\big(2\alpha mI-\alpha(\eta_2+\eta_3)I-\alpha\beta A^\prime A-\alpha\rho(B)I\big)(x^{k+1,T}-x^*)\\&+(1-\frac{\alpha L^2}{\eta_2})\norm{x^{k+1,T}-x^{k+1,T-1}}^2
+\alpha\rho(B)(1-\frac{\rho(B)}{\eta_3})\norm{x^k-x^*}^2
\\&\leq\norm{x^{k+1,T-1}-x^*}^2-c_1\norm{x^{k+1,T}-x^*}^2+\frac{\alpha}{\beta}\norm{\lambda^k-\lambda^*}^2\\&-\frac{\alpha}{\beta}\norm{\lambda^{k+1}-\lambda^*}^2+\alpha\rho(B)\norm{x^k-x^*}^2\end{align*}
\end{proof}
\begin{lemma}\label{lemma:lambdabound-G}
Consider the primal-dual iteration as in \cref{alg:FlexPD-G}, for any $\bar{d}, \bar{c}, \bar{g}, \bar{e}>1$ we have
 \begin{align*}&\norm{\lambda^{k+1} -\lambda ^*}^2\leq\frac{\bar{d}\bar{c}}{\alpha^2s(AA^\prime)}\big(\bar{g}+\frac{\bar{g}\alpha^2L^2}{\bar{g}-1}\big)\norm{x^{k+1,T}-x^{k+1,T-1}}^2+\\&\frac{\bar{d}\bar{c}}{s(AA^\prime)(\bar{c}-1)}\big(\bar{e}\beta^2\rho^2(A^\prime A)+\frac{\bar{e} L^2}{\bar{e}-1}\big)\norm{x^{k+1,T}-x^*}^2+\frac{\bar{d}\rho^2(B)}{
 s(AA^\prime)(\bar{d}-1)}\norm{x^k-x^*}^2.\end{align*}
\end{lemma}
\begin{proof}
We recall the following relation from  \cref{lemma:gradF-G},
	\begin{align*}&\alpha  A^\prime (\lambda^{k+1} -\lambda ^*)=  (x^{k+1,T-1}-x^{k+1,T}) \\&-\alpha B(x^k-x^*)+ \alpha\beta  A^\prime A (x^{k+1,T} - x^*) - \alpha \big(\nabla f(x^{k+1,T-1})- \nabla f(x^*)\big).\end{align*}	
We then add and subtract a term of $\nabla f(x^{k+1,T})$ to the right hand side of the above equality and take square norm of both sides to obtain
\begin{align*}&\norm{\alpha  A^\prime (\lambda^{k+1} -\lambda ^*)}^2= \big\Vert(x^{k+1,T-1}-x^{k+1,T}) -\alpha B(x^k-x^*)+ \alpha\beta  A^\prime A (x^{k+1,T} - x^*)  \\&-\alpha \big(\nabla f(x^{k+1,T-1})- \nabla f(x^{k+1,T})\big)-\alpha \big(\nabla f(x^{k+1,T})- \nabla f(x^*)\big)\big\Vert^2.\end{align*}
By applying the result of \cref{lemma:NormSquare} and by using the Lipschitz property of $\nabla f$, we have for any scalars $\bar{d}, \bar{c}, \bar{g}, \bar{e}>1$, 
\begin{align*}&\norm{\alpha  A^\prime (\lambda^{k+1} -\lambda ^*)}^2 \leq\frac{\bar{d}\alpha^2\rho^2(B)}{\bar{d}-1}\norm{x^k-x^*}^2+ \bar{d}\Bigg[\bar{c}\big\Vert (x^{k+1,T-1}-x^{k+1,T})-  \\&\alpha\big(\nabla f(x^{k+1,T-1}) -\nabla f(x^{k+1,T})\big) \big\Vert^2+\frac{\bar{c}}{\bar{c}-1}\big\Vert \alpha\beta  A^\prime A (x^{k+1,T} - x^*) -\\&\alpha \big(\nabla f(x^{k+1,T})-\nabla f(x^*)\big)\big\Vert ^2\Bigg]\leq\bar{d}\bar{c}\big(\bar{g}+\frac{\bar{g}\alpha^2L^2}{\bar{g}-1}\big)\norm{x^{k+1,T}-x^{k+1,T-1}}^2\\&+\frac{\bar{d}\bar{c}}{\bar{c}-1}\big(\bar{e}\alpha^2\beta^2\rho^2(A^\prime A)+\frac{\bar{e}\alpha^2 L^2}{\bar{e}-1}\big)\norm{x^{k+1,T}-x^*}^2+\frac{\bar{d}\alpha^2\rho^2(B)}{\bar{d}-1}\norm{x^k-x^*}^2.
\end{align*}

Since $\lambda^0=0$ and $\lambda^{k+1} = \lambda^k + \beta  Ax^{k+1,T}$, we have that $\lambda^k$ is in the column space of $A$ and hence orthogonal to the null space of $A^\prime $, hence we have
$\norm{\alpha  A^\prime (\lambda^{k+1} -\lambda ^*)}^2\geq \alpha ^2s(AA^\prime )\norm{\lambda^{k+1}-\lambda^*}^2$.
 By using this inequality and the Lipschitz property of $\nabla f(.)$, we have
 \begin{align*}&\norm{\lambda^{k+1} -\lambda ^*}^2\leq\frac{\bar{d}\bar{c}}{\alpha^2s(AA^\prime)}\big(\bar{g}+\frac{\bar{g}\alpha^2L^2}{\bar{g}-1}\big)\norm{x^{k+1,T}-x^{k+1,T-1}}^2+\frac{\bar{d}\bar{c}}{s(AA^\prime)(\bar{c}-1)}\\&\big(\bar{e}\beta^2\rho^2(A^\prime A)+\frac{\bar{e} L^2}{\bar{e}-1}\big)\norm{x^{k+1,T}-x^*}^2+\frac{\bar{d}\alpha^2\rho^2(B)}{\alpha^2s(AA^\prime)(\bar{d}-1)}\norm{x^k-x^*}^2.\end{align*}
\end{proof}
\begin{theorem}\label{thm:delta-G}
Consider the primal-dual iteration in \cref{alg:FlexPD-G} and recall the definition of $c_1=1+\alpha\rho(B)$. If we choose $\eta_2>0$ and $\eta_3>\rho(B)$ such that $\eta_2+\eta_3<2m-\rho(B)$, then for
  $\alpha<\frac{\eta_2}{L^2}$ and $\beta<\frac{2m-(\eta_2+\eta_3)-\rho(B)}{\rho(A^\prime A)},$ there exists $\delta_G>0$ such that 
\begin{align*}
&c_1\norm{x^{k+1,T}-x^*}^2+\frac{\alpha}{\beta}\norm{\lambda^{k+1}-\lambda^*}\leq\\&\frac{1}{1+\delta_G}\Big(\norm{x^{k+1,T-1}-x^*}^2+\frac{\alpha}{\beta}\norm{\lambda^{k}-\lambda^*}+\alpha\rho(B)\norm{x^k-x^*}^2\Big)
\end{align*}.
\end{theorem}
\begin{proof}
To show the result, we will show that
\begin{align*}
&\delta_G\Big(c_1\norm{x^{k+1,T}-x^*}^2+\frac{\alpha}{\beta}\norm{\lambda^{k+1}-\lambda^*}\Big)\leq\norm{x^{k+1,T-1}-x^*}^2+\frac{\alpha}{\beta}\\&\norm{\lambda^{k}-\lambda^*}+\alpha\rho(B)\norm{x^k-x^*}^2-c_1\norm{x^{k+1,T}-x^*}^2-\frac{\alpha}{\beta}\norm{\lambda^{k+1}-\lambda^*}.
\end{align*}
for some $\delta_G>0$. By comparing the above inequality to the result of \cref{lemma:strcvx-G}, it suffices to show that there exists a $\delta_G>0$ such that
\begin{align*}
&\delta_G\Big(c_1\norm{x^{k+1,T}-x^*}^2+\frac{\alpha}{\beta}\norm{\lambda^{k+1}-\lambda^*}\Big)\leq\\& \norm{x^{k+1,T}-x^*}^2_{P_G}+c_2\norm{x^{k+1,T}-x^{k+1,T-1}}^2
+c_3\norm{x^k-x^*}^2.
\end{align*}
We next collect the terms and focus on showing $\norm{\lambda^{k+1}-\lambda^*}^2\leq$
\begin{equation*}\begin{aligned}\label{ineq:linearRateToShow-G} &  \frac{\beta }{\delta_G\alpha }\norm{x^{k+1,T}-x^*}^2_{P_G-\delta_Gc_1I} +\frac{\beta c_2}{\delta_G\alpha}\norm{x^{k+1,T}-x^{k+1,T-1}}^2
+\frac{\beta c_3}{\delta_G\alpha}\norm{x^k-x^*}^2.\end{aligned}\end{equation*}
 We compare this with the result of  \cref{lemma:lambdabound-G}, and we need to have for some $\delta_G>0$
  \begin{align*}&\frac{\beta c_2}{\delta_G\alpha}\geq \frac{\bar{d}\bar{c}}{\alpha^2s(AA^\prime)}\big(\bar{g}+\frac{\bar{g}\alpha^2L^2}{\bar{g}-1}\big),\quad
  \frac{\beta c_3}{\alpha\delta_G}{\eta_3}\geq\frac{\bar{d}\rho^2(B)}{
 s(AA^\prime)(\bar{d}-1)},\\&\frac{\beta }{\delta_G\alpha }\big(P_G-\delta_Gc_1I\big)\succcurlyeq\frac{\bar{d}\bar{c}}{s(AA^\prime)(\bar{c}-1)}\big(\bar{e}\beta^2\rho^2(A^\prime A)+\frac{\bar{e} L^2}{\bar{e}-1}\big)I.
\end{align*}
for any $\bar{d}, \bar{c}, \bar{g}, \bar{e}>1$. We can find the parameters $\bar{g}$ and $\bar{e}$ to make the right hand side smallest, which would give us the most freedom to choose algorithm parameters.  The term $\frac{\bar{g}}{\bar{g}-1} \alpha^2L^2 +\bar{g}$ is convex in $\bar{g}$ and to minimize it we set derivative to 0 and have 
$ \bar{g} = 1+\alpha L.$
Similarly, we choose
$ \bar{e} =1 + \frac{L}{\beta\rho(A^\prime A)}.$
With these parameter choices, we have
$\bar{g}+\frac{\bar{g}\alpha^2L^2}{\bar{g}-1}=(1+\alpha L)^2$ and   $\bar{e}\beta^2\rho^2(A^\prime A)+\frac{\bar{e} L^2}{\bar{e}-1}=\big(\beta\rho(A^\prime A)+L\big)^2.
$
Now, by considering the definitions of $P_G$, $c_1$, $c_2$, and $c_3$ from \cref{lemma:strcvx-G}, the desired relation can be expressed as
\[\frac{\beta}{\delta_G\alpha}(1-\frac{\alpha L^2}{\eta_2})\geq \frac{\bar{d}\bar{c}}{\alpha^2s(AA^\prime)}(1+\alpha L)^2,\quad\frac{\beta\rho(B)}{\delta_G}(1-\frac{\rho(B)}{\eta_3})\geq\frac{\bar{d}\rho^2(B)}{
 s(AA^\prime)(\bar{d}-1)},\] and  \[\frac{\beta }{\delta_G\alpha }\Big(2\alpha m-\alpha(\eta_2+\eta_3)-\alpha\beta \rho(A^\prime A)-\alpha\rho(B)-\delta_G\big(1+\alpha \rho(B)\big)\Big)\geq\frac{\bar{d}\bar{c}}{s(AA^\prime)(\bar{c}-1)}\big(\beta\rho(A^\prime A)+L\big)^2.
\]
  We next solve the last three inequalities for $\delta_G$ and have
  \begin{align*}
  &\delta_G\leq\frac{\alpha\beta s(AA^\prime)(1-\frac{\alpha L^2}{\eta_2})}{\bar{d}\bar{c}(1+\alpha L)^2},\quad\delta_G\leq\frac{\beta s(AA^\prime)(\bar{d}-1)(1-\frac{\rho(B)}{\eta_3})}{\bar{d}\rho(B)},\\&\delta_G\leq\frac{\beta\big(2m-(\eta_2+\eta_3)-\beta\rho(A^\prime A)-\rho(B)\big)}{\frac{\bar{d}\bar{c}}{s(AA^\prime)(\bar{c}-1)}\big(\beta\rho(A^\prime A)+L\big)^2+\frac{\beta}{\alpha}\big(1+\alpha \rho(B)\big)}.
  \end{align*}
  The right hand side of the above inequalities are positive for $0<\eta_2+\eta_3<2m-\rho(B)$ and
  $\alpha<\frac{\eta_2}{L^2},\quad \rho(B)<\eta_3,\quad\beta<\frac{2m-(\eta_2+\eta_3)-\rho(B)}{\rho(A^\prime A)}.$ 
  The parameter set is nonempty and the proof is complete.
\end{proof}
\begin{theorem}\label{thm:linConv-G} Consider the primal-dual iteration in \cref{alg:FlexPD-G} with $T>1$, recall the definition of $\Gamma_G$, $c_1$, $\eta_2,\eta_3>0$, and $\delta_G$  from \cref{lemma:MBound-G} and \cref{thm:delta-G}, and define $ z^k=\begin{bmatrix} x^k \\ \lambda^k \end{bmatrix},$ and $\mathcal{G}_G=\begin{bmatrix} c_1I & \textbf{0} \\ \textbf{0} & \frac{\alpha}{\beta} I \end{bmatrix}.$ Then for any $\eta_2>0$ and $\eta_3>\rho(B)$ with $\eta_2+\eta_3<2m-\rho(B)$ and for any $0<\tilde{\delta}_G<\delta_G$, if  $\beta<\frac{2m-(\eta_2+\eta_3)-\rho(B)}{\rho(A^\prime A)}$ and $\alpha<\min\Big\lbrace{\frac{\eta_2}{L^2}, \frac{(1+\tilde{\delta}_G)^{\frac{1}{2(T-1)}}-1}{L}, \frac{(1+\tilde{\delta}_G)^{\frac{1}{T-1}}-1}{\beta \rho(AA^\prime)},\frac{(1+\tilde{\delta}_G)^{\frac{1}{T-1}}-1}{\rho(B)}\Big\rbrace}$, we have
\[\norm{z^{k+1}-z^*}^2_{\mathcal{G}_G}\leq\frac{\Gamma_F^{T-1}}{1+\tilde{\delta}_G}\norm{z^k-z^*}^2_{\mathcal{G}_G},\quad\mbox{with}\quad\frac{\Gamma_G^{T-1}}{1+\tilde{\delta}_G}<1.\]
that is $\norm{z^k-z^*}_{\mathcal{G}_G}$ converges Q-linearly to $0$  and consequently $\norm{x^k-x^*}$ converges R-linearly to $0$.
\end{theorem}
\begin{proof}
We note that with $\eta_2+\eta_3<2m-\rho(B)$, $\eta_3>\rho(B)$ $\beta<\frac{2m-(\eta_2+\eta_3)-\rho(B)}{\rho(A^\prime A)}$, and $\alpha<\frac{\eta_2}{L^2}$ the result of \cref{thm:delta-G} holds and we have for every $\tilde{\delta}_G<\delta_G$
\begin{align*}   &c_1\norm{x^{k+1,T} -x^*}^2+ \frac{\alpha }{\beta }\norm{\lambda^{k+1}-\lambda^*}^2\leq  \frac{1}{1+\tilde{\delta}_G}\Big(\norm{x^{k+1,T-1} -x^*}^2 + \frac{\alpha }{\beta }
	\\& \norm{\lambda^k-\lambda^*}^2+\alpha\rho(B)\norm{x^k-x^*}^2\Big)\leq \frac{\Gamma_G^{T-1}}{1+\tilde{\delta}_G}\left(c_1\norm{x^{k} -x^*}^2 
	+ \frac{\alpha }{\beta }\norm{\lambda^k-\lambda^*}^2 \right),
\end{align*}
where we used the result of \cref{lemma:MBound-G} in deriving the second inequality.
Finally, for $T>1$ we need to show that $\frac{\Gamma_G^{T-1}}{1+\tilde{\delta}_G}<1.$
We note that from \cref{lemma:MBound-G}, we have
\[\Gamma_G=\max\Big\lbrace{\bar{p}(1+
 \alpha L)^2, 1+\frac{\bar{p}\bar{q}\alpha\beta\rho(AA^\prime)}{\bar{p}-1}, 1+\frac{\bar{p}\bar{q}\alpha\rho(B)}{(\bar{p}-1)(\bar{q}-1)}\Big\rbrace},\] with $\bar{p}, \bar{q}>1$.
If $\Gamma_G=\bar{p}(1+\alpha L)^2$, we need $\frac{\big(\bar{p}(1+\alpha L)^2\big)^{T-1}}{1+\tilde{\delta}_G}<1,$
which is equivalent to $\bar{p}(1+\alpha L)^2<(1+\tilde{\delta}_G)^{\frac{1}{T-1}}.$
Hence, we have
$\bar{p}<\frac{(1+\tilde{\delta}_G)^{\frac{1}{T-1}}}{(1+\alpha L)^2}.$
Using the fact that $\bar{p}>1$, we need
$\frac{(1+\tilde{\delta}_G)^{\frac{1}{T-1}}}{(1+\alpha L)^2}>1,$
which is equivalent to having $\alpha<\frac{(1+\tilde{\delta}_G)^{\frac{1}{2(T-1)}}-1}{L}.$
We next consider the case with $\Gamma_G=1+\frac{\bar{p}\bar{q}\alpha\beta\rho(AA^\prime)}{\bar{p}-1}$. We need $\frac{\Gamma_G^{T-1}}{1+\tilde{\delta}_G}<1$ and therefore, $\frac{\bar{p}\bar{q}}{\bar{p}-1}<\frac{(1+\tilde{\delta}_G)^{\frac{1}{T-1}}-1}{\alpha\beta\rho(AA^\prime)}$. Given that with any choice of $\bar{p},\bar{q}>1$, we have $\frac{\bar{p}\bar{q}}{\bar{p}-1}>1$, we need $\frac{(1+\tilde{\delta}_G)^{\frac{1}{T-1}}-1}{\alpha\beta\rho(AA^\prime)}>1$, which is satisfied if $\alpha<\frac{(1+\tilde{\delta}_G)^{\frac{1}{T-1}}-1}{\beta\rho(AA^\prime)}$. Similarly, we can consider the other possible value of $\Gamma_G$ and derive the other upper bound on $\alpha$.
\end{proof}
\subsection{Convergence Analysis of FlexPD-C}\label{Cproof}
In order to analyze the convergence properties of FlexPD-C algorithm, we first rewrite the primal update in \cref{alg:FlexPD-C} in the following compact form  
\be 
x^{k+1}=(I-\alpha B)^Tx^k-\alpha C\nabla f(x^k)-\alpha CA^\prime \lambda^k, \label{eq:xUpdate-C}
\ee
 where $C=\sum_{t=0}^{T-1}(I-\alpha B)^t.$  We next proceed to prove the linear convergence rate for our proposed framework. In  \cref{lemma:Cprop-C} and \cref{lemma:gradF-C} we establish some key relations which we use to derive two fundamental inequalities in  \cref{FundamentalIneq-C} and \cref{lambdabound-C}. Finally we use these key inequalities to prove the global linear rate of convergence in  \cref{thm:linconv-C}. \fm{In the following analysis,  we define matrices $M$ and $N$ as follows \be M=C^{-1}(I-\alpha B)^T\quad\mbox{and}\quad N=\frac{1}{\alpha }(C^{-1}-M). \label{MNdef-C}\ee} In the next lemma we show that matrix $C$ is invertible and thus matrices $M$ and $N$ are well-defined.
 \begin{lemma}\label{lemma:Cprop-C} 
 Consider the symmetric positive semi-definite matrix $B$ and matrices $C$, $M$, and $N$. If we choose $\alpha$ such that $I-\alpha B$ is positive definite, i.e., $\alpha <\frac{1}{\rho(B)}$, then matrix $C$ is invertible and symmetric, matrix $N$ is symmetric positive semi-definite, and matrix $M$ is symmetric positive definite with $\frac{\big(1-\alpha \rho(B)\big)^T}{\sum_{t=0}^{T-1}\big(1-\alpha \rho(B)\big)^t}I\preceq M\preceq\frac{1}{T}I.$
\end{lemma} 
\begin{proof} 
 Since $I-\alpha B$ is symmetric, it can be written as $I-\alpha B=VZV^\prime $, where $V\in\mathbb{R}^{n\times n}$ is an orthonormal matrix, i.e., $VV^\prime=I$, whose $i^{th}$ column $v_i$ is the eigenvector of $(I-\alpha B)$ and $v_i^\prime v_t=0$ for $i\neq t$ and $Z$ is the diagonal matrix whose diagonal elements, $Z_{ii}=\mu_i >0$, are the corresponding eigenvalues. We also note that since $V$ is an orthonormal matrix, we have $(I-\alpha B)^t=VZ^tV^\prime$. Therefore,
 \begin{equation*}C=\sum_{t=0}^{T-1}(I-\alpha B)^t=V\big(\sum_{t=0}^{T-1}Z^t\big)V^\prime =V\bar{Z}V^\prime.\end{equation*}
 Hence, matrix $C$ is symmetric. We note that matrix $\bar{Z}$ is a diagonal matrix with $\bar{Z}_{ii}=1+\sum_{t=1}^{T-1}\mu_i ^t$. Since $\mu_i >0$ for all $i$, $\bar{Z}_{ii}\neq 0$ and thus $\bar{Z}$ is invertible and we have 
 $C^{-1}=V\bar{Z}^{-1}V^\prime.$ We also have $M =C^{-1}(I-\alpha  B)^T=V\bar{Z}^{-1}V^\prime VZ^TV^\prime =V\bar{Z}^{-1}Z^TV^\prime =VWV^\prime$,
 where $W$ is a diagonal matrix with $W_{ii}=\frac{\mu_i ^T}{1+\sum_{t=1}^{T-1}\mu_i ^t}$, consequently, matrix $M$ is symmetric.
 We next find the smallest and largest eigenvalues of matrix $M$. We note that since $W_{ii}$ is increasing in $\mu_i $, the smallest and largest eigenvalues of $M$ can be computed using the smallest and largest eigenvalues of $I-\alpha B$. We have $0\preceq B\preceq \rho(B)I$, where $\rho(B)$ is the largest eigenvalue of matrix $B$. Therefore, the largest and smallest eigenvalues of $I-\alpha B$ are $1$ and $1-\alpha \rho(B)$ respectively. Hence, $\frac{\big(1-\alpha \rho(B)\big)^T}{\sum_{t=0}^{T-1}\big(1-\alpha \rho(B)\big)^t}\preceq M\preceq\frac{1}{T}.$ We next use the eigenvalue decomposition of matrices $C^{-1}$ and $M$ to obtain $C^{-1}-M=V\bar{Z}^{-1}V^\prime -VWV^\prime =V(\bar{Z}^{-1}-W)V^\prime,$
 where $\bar{Z}^{-1}-W$ is a diagonal matrix, and its $i^{th}$ diagonal element is equal to $\frac{1-\mu_i ^T}{1+\sum_{t=1}^{T-1}\mu_i ^t}$. Since $0<\mu_i \leq 1$ for all $i$, we have $\frac{1-\mu_i ^T}{1+\sum_{t=1}^{T-1}\mu_i ^t}\geq 0$ and hence $N$ is symmetric positive semi-definite.
 \end{proof}

\begin{lemma}\label{lemma:gradF-C} Consider the primal-dual iterates as in \cref{alg:FlexPD-C} and recall the definitions of matrices $M$ and $N$ from Eq. \cref{MNdef-C}, if $\alpha <\frac{1}{\rho(B)}$, then 
	\begin{align*} &\alpha (\nabla f(x^k)- \nabla f(x^*) )=  M(x^k-x^{k+1}) +\\& \alpha  (\beta  A^\prime A-N) (x^{k+1} - x^*) -\alpha  A^\prime (\lambda^{k+1} -\lambda ^*)\end{align*}
\end{lemma}
\begin{proof}
	At each iteration, from Eq. \cref{eq:xUpdate-C} we have
$\alpha C\nabla f(x^k)  = (I-\alpha B)^Tx^k-x^{k+1} -\alpha C A^\prime  \lambda^k.$
	Moreover, from Eq. \cref{eq:lambda} we have
	$\lambda ^k = \lambda ^{k+1} - \beta  A x^{k+1}.$ We can substitute this expression for $\lambda^k$ into the previous equation and have
	\begin{equation}\begin{aligned} &\alpha C\nabla f(x^k)  =  (I-\alpha B)^Tx^k-x^{k+1} -\alpha C A^\prime  (\lambda ^{k+1} - \beta  A x^{k+1}) =(I-\alpha  B)^T\\&(x^k-x^{k+1})  +\big(\alpha  \beta  CA^\prime A-I+(I-\alpha  B)^T\big) x^{k+1} -\alpha C A^\prime \lambda^{k+1},\label{gradeq-C}\end{aligned}\end{equation} where we added and subtracted a term of $(I-\alpha  B)^T x^{k+1}$.
Since an optimal solution pair $(x^*, \lambda^*)$ is a fixed point of the algorithm update, we also have
	\[ \alpha  C\nabla f(x^*)  =   \big(\alpha \beta  CA^\prime A-I+(I-\alpha  B)^T\big) x^{*} -\alpha  C A^\prime \lambda^{^*}.\]
	We then subtract the above inequality from Eq. \cref{gradeq-C} and multiply both sides by $C^{-1}$ [c.f. \cref{lemma:Cprop-C}],  to obtain the result.
\end{proof}
\begin{lemma} \label{FundamentalIneq-C}
Consider the primal-dual iterates as in \cref{alg:FlexPD-C} and recall the definition of matrices $M$ and $N$ from Eq. \cref{MNdef-C}. If $\alpha<\frac{1}{\rho(B)}$, we have 
\ew{for any $\eta_4>0$},
\begin{equation*}\begin{aligned}  &\norm{x^{k+1}-x^*}^2_{P_C}+ \norm{x^{k+1}-x^k}^2_{Q_C}\leq\\& 
\norm{x^k-x^*}^2_M - \norm{x^{k+1}-x^*}^2_M
+ \frac{\alpha }{\beta }\left(\norm{\lambda^k-\lambda^*}^2 - \norm{\lambda^{k+1}-\lambda^*}^2\right)
 \end{aligned}\end{equation*}
 with $P_C=2\alpha  m I-\alpha \eta_4 I +2\alpha N-\alpha\beta  A^\prime A$ and
 $Q_C=M- \frac{\alpha  L^2}{\eta_4}I$
\end{lemma}
\begin{proof}
From strong convexity of function $f(x)$, we have
	\begin{align*} & 2\alpha m\norm{x^{k+1}-x^*}^2 \leq 2\alpha  (x^{k+1}-x^*)^\prime \big(\nabla f(x^{k+1}) - \nabla f(x^*)\big) \\&= 2\alpha  (x^{k+1}-x^*)^\prime \big(\nabla f(x^{k+1})-\nabla f(x^k)\big)+ 2\alpha  (x^{k+1}-x^*)^\prime (\nabla f(x^k)-\nabla f(x^*)\big),\end{align*} where we add and subtract a term $(x^{k+1}-x^*)^\prime \nabla f(x^k)$.
	We can substitute the equivalent expression of $ \alpha (\nabla f(x^k)- \nabla f(x^*) )$ from  \cref{lemma:gradF-C} and have
		\begin{equation}\begin{aligned} & 2\alpha  m \norm{x^{k+1}-x^*}^2\leq2\alpha  (x^{k+1}-x^*)^\prime \big(\nabla f(x^{k+1})-\nabla f(x^k)\big)+\\ &2\alpha  (x^{k+1}-x^*)^\prime  (\beta  A^\prime A - N)(x^{k+1}-x^*)+2(x^{k+1}-x^*)^\prime M(x^k-x^{k+1})  \\&-2\alpha (x^{k+1}-x^*)^\prime  A^\prime (\lambda^{k+1} -\lambda ^*). \label{eq:update1-C} \end{aligned}\end{equation}
		By Young's inequality and the Lipschitz continuity of $\nabla f(.)$ we have
 $2\alpha  (x^{k+1}-x^*)^\prime \big(\nabla f(x^{k+1})-\nabla f(x^k)\big) \leq \alpha \eta_4 \norm{x^{k+1}-x^*}^2+\frac{\alpha  L^2 }{\eta_4}\norm{x^{k+1}-x^k}^2$, for all $\eta_4>0$. 
By dual update Eq. \cref{eq:lambda} and feasibility of $x^*$, we have 
$Ax^{k+1} = \frac{1}{\beta }(\lambda^{k+1}-\lambda^k),\quad Ax^* = 0.$
These two equations combined yields
$\alpha (x^{k+1}-x^*)^\prime  A^\prime (\lambda^{k+1} -\lambda ^*) = \frac{\alpha }{\beta }(\lambda^{k+1}-\lambda^k)^\prime (\lambda^{k+1}-\lambda^*).$
 We now focus on the last two terms of Eq. \cref{eq:update1-C}. First since matrix $M$ is symmetric, we have
$-2(x^{k+1}-x^*)^\prime M(x^{k+1}-x^k)  = \norm{x^k-x^*}^2_M -\norm{x^{k+1}-x^*}^2_M - \norm{x^{k+1}-x^k}^2_M.$
 Similarly, we have $-2\frac{\alpha }{\beta }(\lambda^{k+1}-\lambda^k)^\prime (\lambda^{k+1}-\lambda^*) = \frac{\alpha }{\beta }\left(\norm{\lambda^k-\lambda^*}^2 - \norm{\lambda^{k+1}-\lambda^*}^2-\norm{\lambda^{k+1}-\lambda^k}^2\right).$
 Now we combine the terms in the preceding three relations and Eq. \cref{eq:update1-C} and have 
\begin{align*}&2\alpha  m \norm{x^{k+1}-x^*}^2\leq 2\alpha \norm{x^{k+1}-x^*}^2_{\beta A^\prime A-N}+ \alpha \eta_4 \norm{x^{k+1}-x^*}^2\\&+\frac{\alpha  L^2}{\eta_4}\norm{x^{k+1}-x^k}^2+ \norm{x^k-x^*}^2_M- \norm{x^{k+1}-x^*}^2_M-\norm{x^{k+1}-x^k}^2_M \\&+ \frac{\alpha }{\beta }\left(\norm{\lambda^k-\lambda^*}^2 - \norm{\lambda^{k+1}-\lambda^*}^2-\norm{\lambda^{k+1}-\lambda^k}^2\right).
 \end{align*} We now use Eq. \cref{eq:lambda} together with the fact that $Ax^*=0$ to obtain $\norm{\lambda^{k+1}-\lambda^k}^2 =\beta ^2(x^{k+1}-x^*)^\prime (A^\prime A)(x^{k+1}-x^*).
$ By substituting this into the previous inequality and by rearranging the terms in the above inequality, we complete the proof.
\end{proof}
\begin{lemma}\label{lambdabound-C}
Consider the primal-dual iterates as in \cref{alg:FlexPD-C} and recall the definition of symmetric matrices $M$ and $N$ from Eq. \cref{MNdef-C} then if $\alpha<\frac{1}{\rho(B)}$, for $\tilde{d}, \tilde{g}, \tilde{e}>1$ we have
\begin{equation*}\begin{aligned}\label{ineq:toCompare-C}
  &\norm{\lambda^{k+1}-\lambda^*}^2\leq \frac{\tilde{d}}{\alpha ^2 s(AA^\prime )} \left(\frac{\tilde{e}}{\tilde{e}-1} \rho(M)^2+\tilde{e} \alpha ^2L^2\right)\norm{x^k-x^{k+1}}^2+\\& \frac{\tilde{d}}{(\tilde{d}-1)\alpha ^2 s(AA^\prime )}\times \left(\frac{\tilde{g}}{\tilde{g}-1} \alpha ^2 \rho \big((\beta  A^\prime A-N)^2\big)+\tilde{g}\alpha ^2L^2\right)\norm{x^{k+1} - x^*}^2,
  \end{aligned}\end{equation*}
with $s(AA^\prime )$ being the smallest nonzero eigenvalue of matrix $AA^\prime$.
\end{lemma}
\begin{proof}
We recall the following relation from  \cref{lemma:gradF-C}, $\alpha  A^\prime (\lambda^{k+1} -\lambda ^*)=$
	\begin{align*} M(x^k-x^{k+1}) + \alpha  (\beta  A^\prime A-N) (x^{k+1} - x^*) - \alpha (\nabla f(x^k)- \nabla f(x^*) ).\end{align*}	
We can add and subtract a term of $\nabla f(x^{k+1})$ and take squared norm of both sides of the above equality to obtain 
\begin{align*}&\norm{\alpha  A^\prime (\lambda^{k+1} -\lambda ^*)}^2 =\big\Vert M(x^k-x^{k+1}) +  \alpha  (\beta  A^\prime A-N) (x^{k+1} - x^*) -\\&\alpha \left((\nabla f(x^{k})- \nabla f(x^{k+1} )\right)- \alpha \left(\nabla f(x^{k+1})- \nabla f(x^*) \right)\big\Vert^2.\end{align*}

By using the result of \cref{lemma:NormSquare}, we have for any $\tilde{d}, \tilde{g}, \tilde{e}>1$
\begin{align*} &\norm{\alpha  A^\prime (\lambda^{k+1} -\lambda ^*)}^2 \leq \tilde{d}\Big(\frac{\tilde{e}}{\tilde{e}-1}\norm{x^k-x^{k+1}}^2_{M^2} +\tilde{e} \alpha ^2\norm{\nabla f(x^{k})- \nabla f(x^{k+1}) }^2\Big) +\\& \frac{\tilde{d}}{\tilde{d}-1}\Big(\frac{\tilde{g}}{\tilde{g}-1} \alpha ^2\norm{x^{k+1} - x^*}^2_{(\beta  A^\prime A-N)^2}+ \tilde{g}\alpha ^2\norm{\nabla f(x^{k+1})- \nabla f(x^*)}^2\Big).\end{align*}

Since $\lambda^0=0$ and $\lambda^{k+1} = \lambda^k + \beta  Ax^{k+1}$, we have that $\lambda^k$ is in the column space of $A$ and hence orthogonal to the null space of $A^\prime $, therefore, we have
$\norm{\alpha  A^\prime (\lambda^{k+1} -\lambda ^*)}^2\geq \alpha ^2s(AA^\prime )\norm{\lambda^{k+1}-\lambda^*}^2$. By using this relation and Lipschitz property of $\nabla f(.)$ , we have
 \begin{align*}
 &\alpha ^2 s(AA^\prime )\norm{\lambda^{k+1}-\lambda^*}^2\leq \tilde{d}\Big((x^k-x^{k+1})^\prime\left[\frac{\tilde{e}}{\tilde{e}-1} M^2+\tilde{e} \alpha ^2L^2I\right](x^k-x^{k+1}) \Big) + \\&\frac{\tilde{d}}{\tilde{d}-1}\Big((x^{k+1} - x^*)^\prime  \left[\frac{\tilde{g}}{\tilde{g}-1} \alpha ^2 (\beta  A^\prime A-N)^2+\tilde{g}\alpha ^2L^2I\right] (x^{k+1} - x^*)\Big).
 \end{align*}
 By using the facts that $M^2\preceq\rho(M^2) I$ and $(\beta  A^\prime A-N)^2\preceq\rho\big((\beta  A^\prime A-N)^2\big)I$, we complete the proof.
  \end{proof}
\begin{theorem} \label{thm:linconv-C}
Consider the primal-dual iterates as in \cref{alg:FlexPD-C}, recall the definition of matrix $M$ from Eq. \cref{MNdef-C}, and define $ z^k=\begin{bmatrix} x^k \\ \lambda^k \end{bmatrix}, \quad \mathcal{G}_C=\begin{bmatrix} M & \textbf{0} \\ \textbf{0} & \frac{\alpha}{\beta} I \end{bmatrix}.$  If the primal and dual stepsizes satisfy 
$0<\beta <\frac{2m-\eta_4}{\rho(A^\prime A)}$, $0<\alpha <\frac{1-\left(\frac{L^2}{L^2+\eta_4\rho(B)}\right)^{1/T}}{\rho(B)}$ with $0<\eta_4<2m$, then there exists a $\delta_C>0$ such that 
\[\norm{z^{k+1}-z^*}^2_{\mathcal{G}_C}\leq\frac{1}{1+\delta_C}\norm{z^k-z^*}^2_{\mathcal{G}_C},\] that is $\norm{z^k-z^*}_{\mathcal{G}_C}$ converges Q-linearly to $0$  and consequently $\norm{x^k-x^*}_M$ converges R-linearly to $0$.
\end{theorem}
\begin{proof}
To show linear convergence, we will show that
\begin{align*}   \delta_C&\left(\norm{x^{k+1}-x^*}^2_ M + \frac{\alpha }{\beta }\norm{\lambda^{k+1}-\lambda^*}^2\right)\leq \norm{x^k-x^*}^2_M \\&- \norm{x^{k+1}-x^*}^2_M
+ \frac{\alpha }{\beta }\left(\norm{\lambda^k-\lambda^*}^2 - \norm{\lambda^{k+1}-\lambda^*}^2\right),
\end{align*}
for some $\delta_C>0$. By using the result of  \cref{FundamentalIneq-C}, it suffices to show that there exists a $\delta_C>0$ such that $\delta_C\left(\norm{x^{k+1}-x^*}^2_M + \frac{\alpha }{\beta }\norm{\lambda^{k+1}-\lambda^*}^2\right) \leq  \norm{x^{k+1}-x^*}^2_{P_C}+\norm{x^{k+1}-x^k}^2_{Q_C}.$

We collect the terms and we will focus on showing 
\begin{equation}
\begin{aligned}\label{ineq:linearRateToShow-C}   &\norm{\lambda^{k+1}-\lambda^*}^2\leq  \frac{\beta }{\delta_C\alpha }\norm{x^{k+1}-x^k}^2_{Q_C}+\frac{\beta }{\delta_C\alpha }\norm{x^{k+1}-x^*}^2_{P_C-\delta_C M}.\end{aligned}\end{equation}
We compare Eq. \cref{ineq:linearRateToShow-C} with the result of \cref{lambdabound-C}, and we need to have for some $\delta_C>0$ 
  \begin{align*}&\frac{\beta }{\delta_C\alpha }Q_C \succcurlyeq \frac{\tilde{d}}{\alpha ^2 s(AA^\prime )} \left(\frac{\tilde{e}\rho(M)^2}{\tilde{e}-1} +\tilde{e} \alpha ^2L^2\right)I,\\
  &\frac{\beta }{\delta_C\alpha }(P_C-\delta_C M)\succcurlyeq \frac{\tilde{d}}{(\tilde{d}-1)\alpha ^2 s(AA^\prime )} \left(\frac{\tilde{g}\alpha ^2 }{\tilde{g}-1}  \rho\big((\beta  A^\prime A-N)^2\big)+\tilde{g}\alpha ^2L^2\right)I.
  \end{align*}
  Since the previous two inequalities holds for all $\tilde{d}, \tilde{g}, \tilde{e}>1$, we can find the parameters $\tilde{e}$ and $\tilde{g}$ to make the right hand sides the smallest, which would give us the most freedom to choose algorithm parameters.  The term $\frac{\tilde{e}}{\tilde{e}-1} \rho(M)^2 + \tilde{e} \alpha ^2L^2$ is convex in $\tilde{e}$ and to minimize it we set derivative to $0$ and have 
$\tilde{e} = 1+ \frac{\rho(M)}{\alpha  L}.$
Similarly, we choose $\tilde{g}$ to be
$ \tilde{g} =1 + \frac{ \sqrt{\rho \big((\beta  A^\prime A-N)^2\big)}}{ L}.$
With these parameter choices, we have
$\left(\frac{\tilde{e}}{\tilde{e}-1} \rho(M)^2+\tilde{e} \alpha ^2L^2\right) = (\rho(M)+ \alpha  L)^2,$ and
$\left(\frac{\tilde{g}}{\tilde{g}-1} \alpha ^2 \rho \big((\beta  A^\prime A-N)^2\big)+\tilde{g}\alpha ^2L^2\right) = \alpha ^2 \Big(\sqrt{\rho\big((\beta  A^\prime A-N)^2\big)}+L\Big)^2.$ By using the definitions of $P_C$ and $Q_C$ from \cref{FundamentalIneq-C}, the desired relations can be expressed as \[\frac{\beta }{\delta_C\alpha }(M- \frac{\alpha  L^2}{\eta_4}I) \succcurlyeq \frac{\tilde{d}}{\alpha ^2 s(AA^\prime )} \big(\rho(M)+ \alpha  L\big)^2I\] and
  \[\frac{\beta }{\delta_C\alpha }(2\alpha  m-\alpha \eta_4 I +2\alpha  N-\alpha \beta  A^\prime A-\delta_C M)\succcurlyeq  \frac{\tilde{d}}{(\tilde{d}-1) s(AA^\prime )}\Big(\sqrt{\rho\big((\beta  A^\prime A-N)^2\big)} +L\Big)^2I.
  \]
By using the fact that $N$ and $A^\prime A$ are  positive semi-definite matrices, and the result of \cref{lemma:Cprop-C} to bound eigenvalues of matrix $M$, the desired relations can be satisfied if

  \begin{align*}&\frac{\beta }{\delta_C\alpha }\Big(\frac{\big(1-\alpha \rho(B)\big)^T}{\sum_{t=0}^{T-1}\big(1-\alpha \rho(B)\big)^t}-\frac{\alpha  L^2}{\eta_4}\Big) \geq \frac{\tilde{d}}{\alpha ^2 s(AA^\prime )} (\frac{1}{T}+ \alpha  L)^2,\\
  &\frac{\beta }{\delta_C\alpha }(2\alpha  m-\alpha \eta_4-\alpha \beta \rho(A^\prime A) -\frac{\delta_C}{T}) \geq  \frac{\tilde{d}}{(\tilde{d}-1) s(AA^\prime )} \Big(\sqrt{\rho\big((\beta  A^\prime A-N)^2\big)}+L\Big)^2.
  \end{align*}
For the first inequality, we can multiply both sides by $\delta_C\alpha $ and rearrange the terms to have
$
\delta_C \leq \frac{\alpha \beta  \Big(\frac{\big(1-\alpha \rho(B)\big)^T}{\sum_{t=0}^{T-1}\big(1-\alpha \rho(B)\big)^t}-\frac{\alpha  L^2}{\eta_4}\Big)s(AA^\prime )  }{\tilde{d}(\frac{1}{T}+ \alpha  L)^2}.
$
We can similarly solve for the second inequality, 
$\delta_C\leq\frac{ \beta (2\alpha  m-\alpha \eta_4-\alpha \beta \rho (A^\prime A) )}{ \frac{\tilde{d}\alpha }{(\tilde{d}-1) s(AA^\prime )}   \Big(\sqrt{\rho\big((\beta  A^\prime A-N)^2\big)}+L\Big)^2+\frac{\beta }{T}}.$
We next show that for suitable choices of $\alpha$ and $\beta$, the upper bounds on $\delta_C$ are both positive. For $0<\beta <\frac{2m}{\rho(A^\prime A)}$ and $0<\eta_4<2m-\beta \rho(A^\prime A)$, the first upper bound on $\delta_C$ is positive. In order for the second upper bound for $\delta_C$ to be positive we need
$\frac{\big(1-\alpha \rho(B)\big)^T}{\sum_{t=0}^{T-1}\big(1-\alpha \rho(B)\big)^t}-\frac{\alpha  L^2}{\eta_4}>0. $
Since $1-\alpha \rho(B)\neq 1$, we have $\sum_{t=0}^{T-1}\big(1-\alpha \rho(B)\big)^t=\frac{1-\big(1-\alpha \rho(B)\big)^T}{1-\big(1-\alpha \rho(B)\big)}=\frac{1-\big(1-\alpha \rho(B)\big)^T}{\alpha \rho(B)}.$ Therefore $\frac{\big(1-\alpha \rho(B)\big)^T}{1-\big(1-\alpha \rho(B)\big)^T}\alpha \rho(B)-\alpha \frac{L^2}{\eta_4}>0,$ which holds true for $0<\alpha <\frac{1-\left(\frac{L^2}{L^2+\eta_4\rho(B)}\right)^{1/T}}{\rho(B)}$.

Hence, the parameter set is nonempty and thus we can find $\delta_C>0$ which establishes linear rate of convergence. 
\end{proof}	
\ff{\begin{remark}\label{Bchoice-C} If we choose $B=\beta A^\prime A$, from the analysis of the FlexPD-F and FlexPD-C algorithms we can see that $\beta$ can be unbounded. For FlexPD-G algorithm, this choice of matrix $B$ together with \cref{ass:Bbound} impose new upper bounds on $\beta$. 
\end{remark}}
\begin{remark}\label{Tchoice}
To find an optimal value for the number of primal updates per iteration, $T$, leading to the best convergence rate for the three algorithms in our framework, we can optimize over various parameters in the analysis. Due to the complicated form of stepsize bounds and the linear rate constants, a general result is quite messy specially for FlexPD-F and FlexPD-G algorithms. However, for FlexPD-C algorithm we can show that the upper bound on $T\alpha$ [c.f. \cref{thm:linconv-C}] is increasing in $T$ and approaches $-\ln{\frac{L^2}{L^2+\eta\rho(B)}}$ for large values of $T$. This suggests that the improvement of convergence speed from increasing $T$ diminishes for large $T$ in FlexPD-C algorithm. In our numerical studies (Section \ref{sec:simul}, we found $T=2, 3$ often gives the best balance for performance and computation/communication costs tradeoff. 
\end{remark}

\begin{remark} The stepsize parameters in our algorithms are common among all agents and computing them requires global variables across the network. These global variables can be obtained by applying a consensus algorithm before the main algorithm \cite{shi2015extra,wu2018decentralized, mokhtari2015network}.
\end{remark}
\begin{figure}
  \centering
  \subfloat{\includegraphics[width=0.33\textwidth]{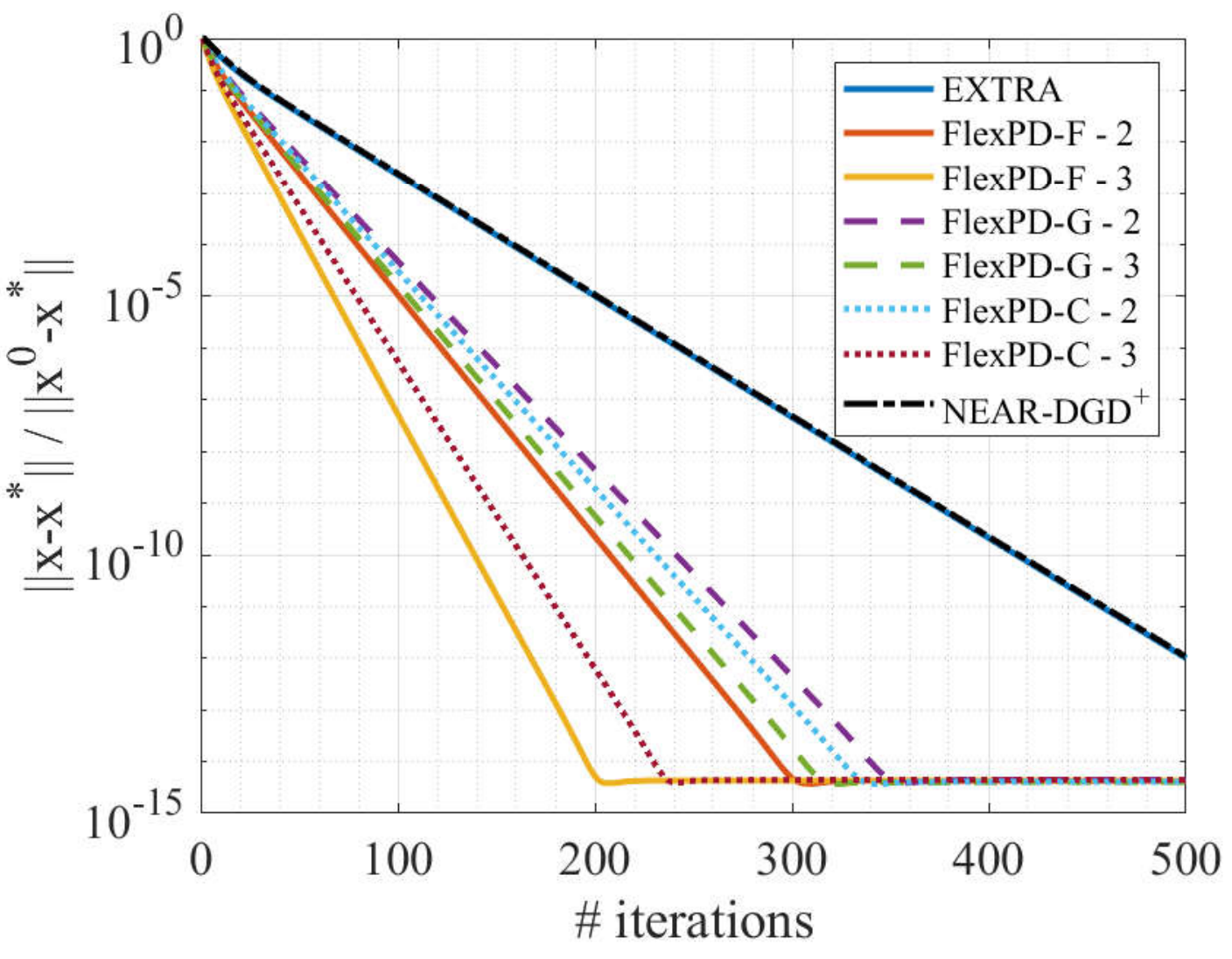}}
  \hfill
  \subfloat{\includegraphics[width=0.33\textwidth]{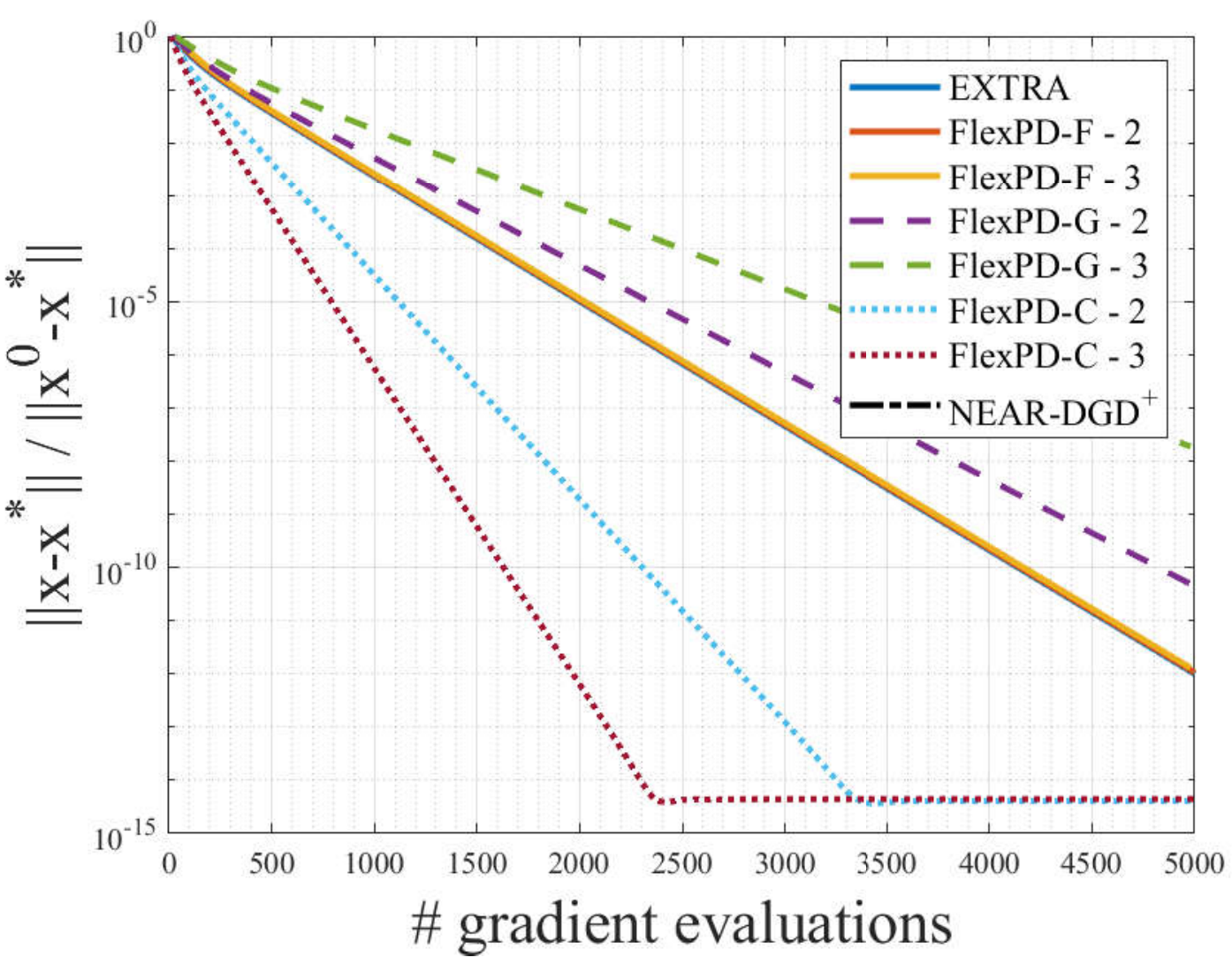}}
  \hfill
  \subfloat{\includegraphics[width=0.33\textwidth]{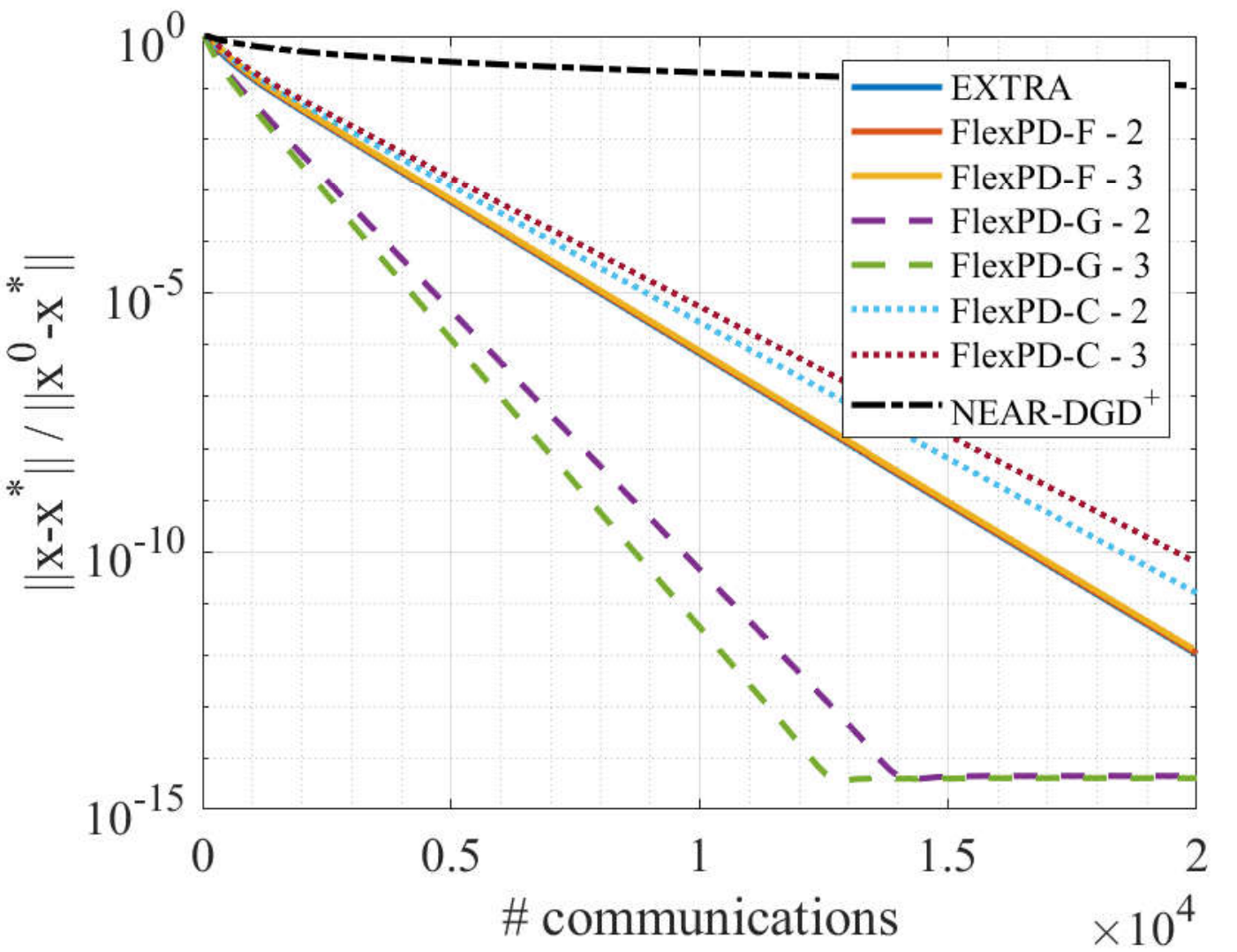}}
  \caption{Performance of FlexPD algorithms with 2 and 3 primal updates per iteration, NEAR-DGD$^+$, and EXTRA in terms of the relative error.} 
  \vskip -0.2in
     \label{fig:logreg}
\end{figure}
\section{Numerical Experiments}\label{sec:simul}
In this section, we present some numerical experiments, where we compare the performance of our proposed algorithms with other first-order methods. \ff{We also study the performance of our framework on networks with different sizes and topologies. In these experiments we simulate FlexPD-C with its theoretical stepsize, due to its explicit form of stepsize bounds.} In all experiments, we set $B=\beta A^\prime A$ for our algorithm. 
\par To compare the performance of our proposed algorithms with other first-order methods, we consider solving a binary classification problem by using regularized logistic regression. We consider a setting where $K$ training samples  are uniformly distributed over $n$ agents in a network with $4-$regular graph, in which agents first form a ring and then each agent gets connected to two other \ew{random} neighbors (one from each side). Each agent $i$ has access to one batch of data with  $k_i=\frac{K}{n}$ samples. This problem can be formulated as
\[\min_x f(x)=\frac{\kappa}{2}\norm{x}^2+\frac{1}{K}\sum_{i=1}^n\sum_{j=1}^{k_i}\log\big[1+\exp(-v_{ij}u_{ij}^\prime x)\big],\]
where $u_{ij}$ and $v_{ij}$, $j\in\{ 1, 2, ..., k_i\}$ are the feature vector and the label for the data point $j$ associated with agent $i$ and the regularizer term $\frac{\kappa}{2}\big\Vert x\big\Vert^2$ is added to avoid overfitting. We can write this objective function in the form of $f(x)=\sum_{i=1}^n f_i(x)$, where $f_i(x)$ is defined as 
$f_i(x)=\frac{\kappa}{2n}\norm{x}^2+\frac{1}{K}\sum_{j=1}^{k_i}\log\big[1+\exp(-v_{ij}u_{ij}^\prime x)\big].$
In our simulations, we use the \textit{diabetes-scale} dataset \cite{chang2011libsvm},  with $768$ data points, distributed uniformly over $10$ agents. Each data point has a feature vector of size $8$ and a label which is either $1$ or $-1$. \fm{In \cref{fig:logreg} we compare the performance of our primal-dual algorithms in \cref{alg:FlexPD-F}, \cref{alg:FlexPD-G}, and \cref{alg:FlexPD-C}, with $T=2, 3$ with two other first-order methods with exact convergence: EXTRA algorithm \cite{shi2015extra}, and NEAR-DGD$^+$ algorithm \cite{berahas2018balancing}, in terms of relative error, $\frac{\norm{x^k-x^*}}{\norm{x^0-x^*}}$, with respect to number of iterations, total number of gradient evaluations, and total number of communications.} To compute the benchmark $x^*$ we used \textit{minFunc} software \cite{schmidt2005minfunc} and the stepsize parameters are tuned for each algorithm using random search. We can see that increasing the number of primal updates improves the performance of the algorithms while incurring a  higher computation or communication cost. \ff{In our experiments, we observe that the performance of FlexPD-F approaches to the method of multipliers by increasing $T$ and carefully choosing stepsizes. This improvement, however, is less in FlexPD-C and FlexPD-G algorithms, due to the effect of the outdated gradients and old information from neighbors.} \fm{EXTRA algorithm is a special case of our framework for specific choices of matrices $A$ and $B$ and one primal update per iteration.} In the NEAR-DGD$^{+}$ the number of communication rounds increases linearly with the iteration number, which explains its slow rate of convergence with respect to the number of communications. We obtained similar results for other standard machine learning datasets, including \textit{mushroom}, \textit{heart-scale}, \textit{a1a}, \textit{australian-scale}, and \textit{german} \cite{chang2011libsvm}. 
\begin{figure} 
  \centering
  \subfloat[]{\includegraphics[width=0.24\textwidth]{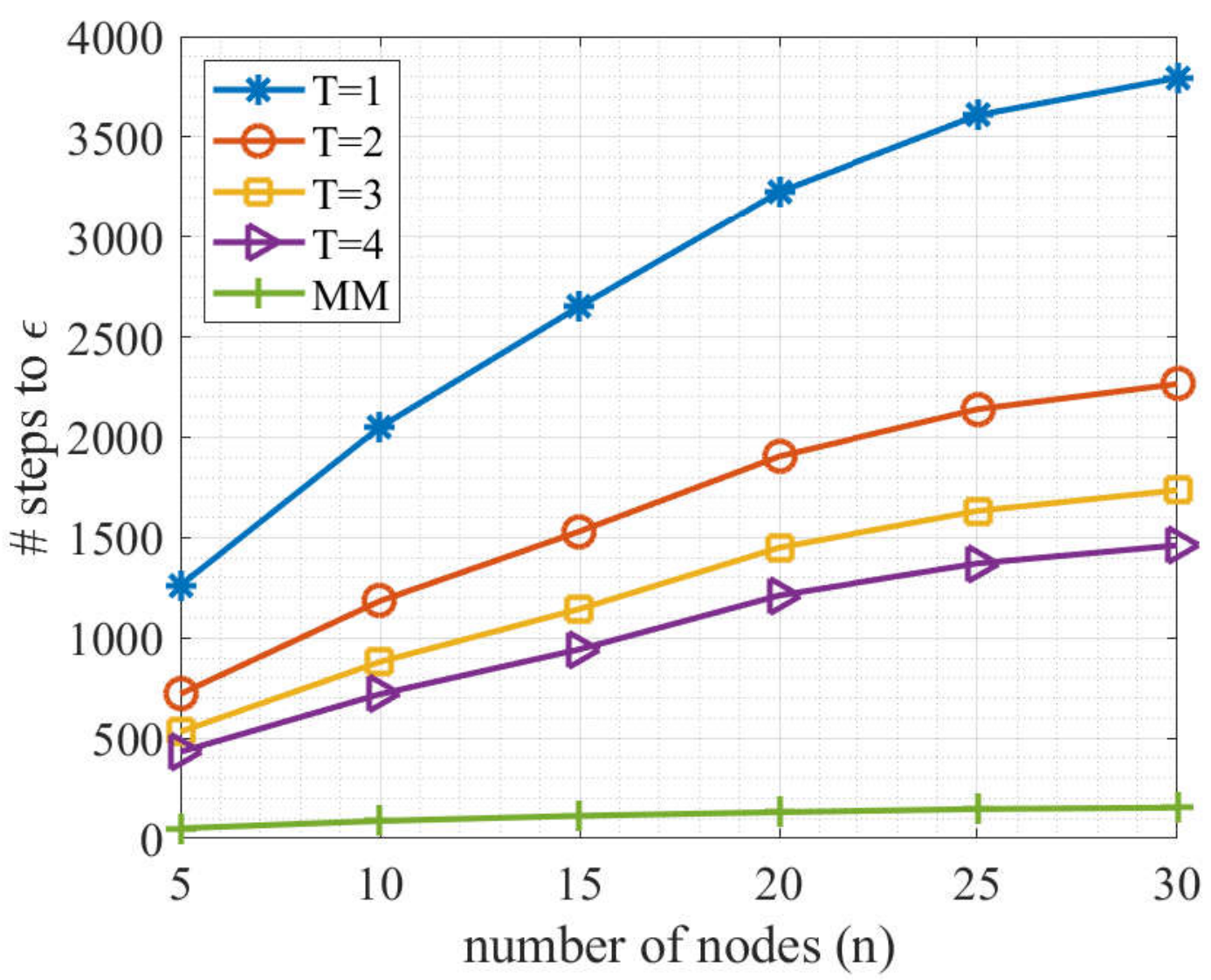}\includegraphics[width=0.24\textwidth]{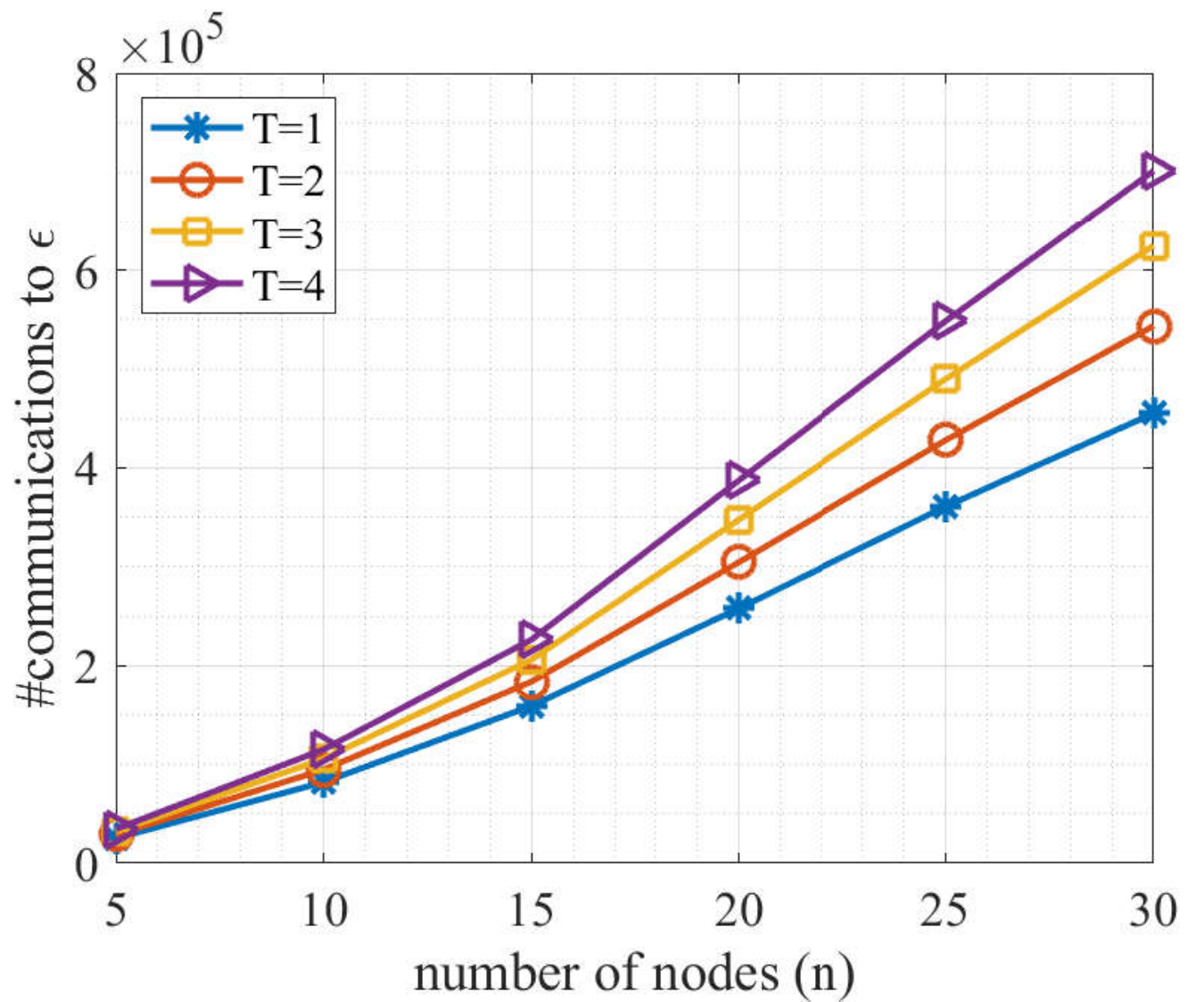}}
  \hfill
  \subfloat[]{\includegraphics[width=0.24\textwidth]{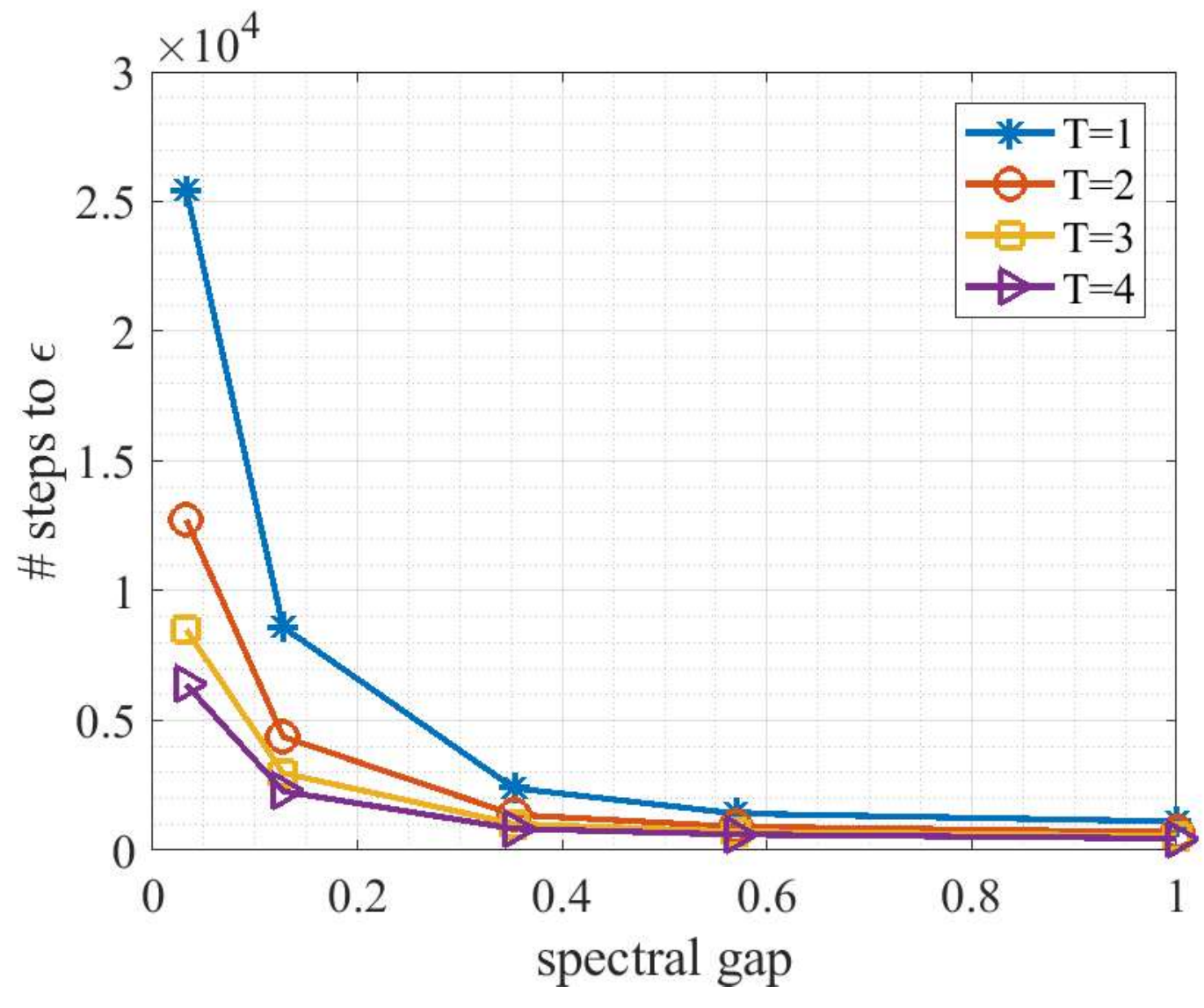}\includegraphics[width=0.24\textwidth]{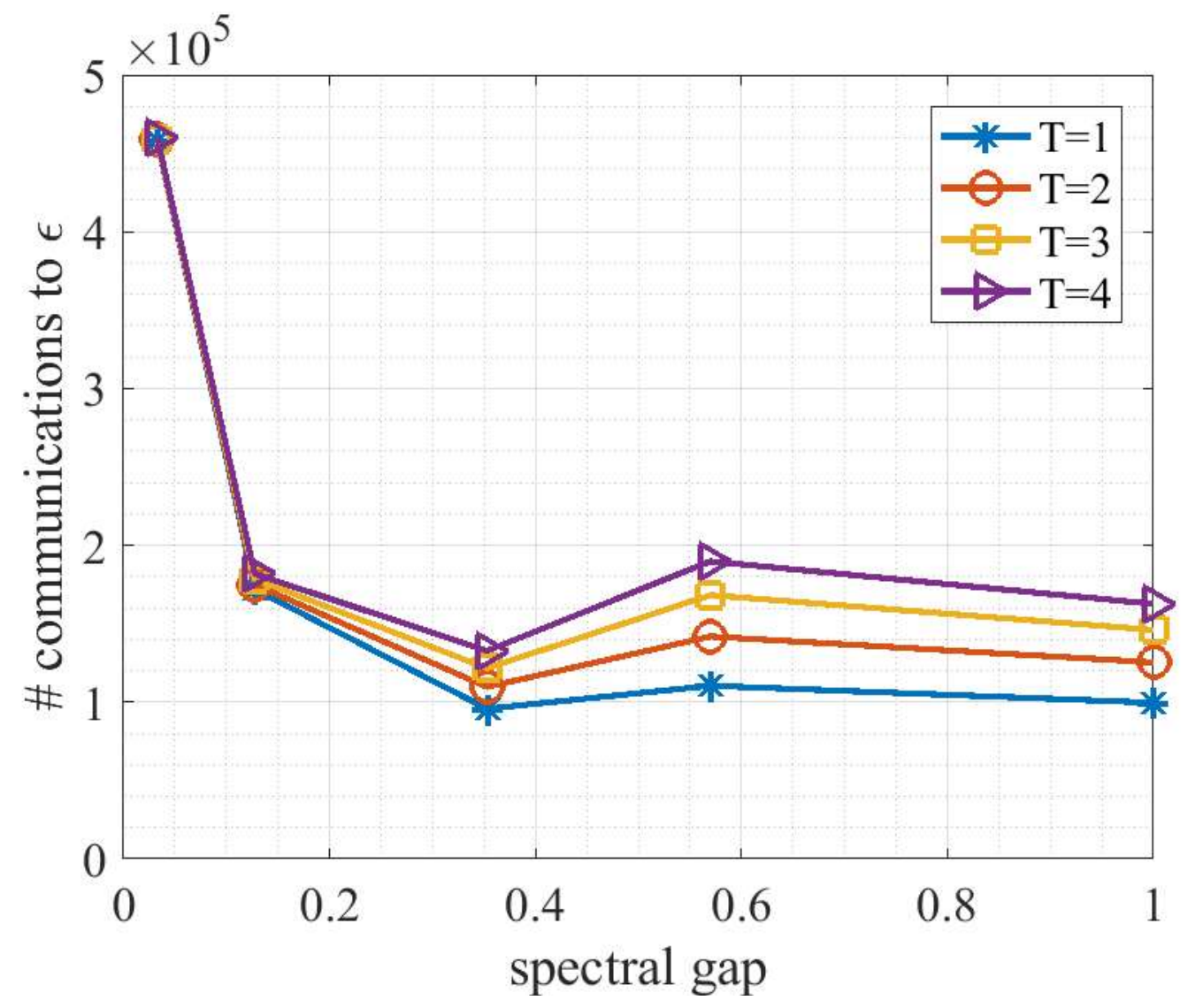}}
  \caption{Performance of FlexPD-C algorithm with T=1,...,4 on different graphs.} 
  \vskip -0.2in
    \label{fig:radii}
\end{figure}
\par To study the performance of FlexPD-C on networks with different sizes, we consider $5-30$ agents, which are connected with random $4-$regular graphs. The objective function at each agent $i$ is $f_i(x)=c_i(x_i-b_i)^2$ with $c_i$ and $b_i$ being random integers chosen from $[1,10^3]$ and $[1,100]$. We simulate the algorithm for $1000$ random seeds and we plot the average number of steps until the relative error is less than $\epsilon=0.01$, i.e., $\frac{\norm{x^k-x^*}}{\norm{x^0-x^*}}<0.01$ in part (a) of \cref{fig:radii}. The centralized implementation of the method of multipliers is also included as a benchmark. The primal stepsize parameter $\alpha$ at each seed is chosen based on the theoretical bound given in \cref{thm:linconv-C} and the dual stepsize is $\beta=T$.  \ew{We observe that as the network size grows, the number of steps to optimality of our proposed method grows sublinearly and the number of communications grows almost linearly.} 

\par To study the performance of FlexPD-C in networks with different topologies, we consider solving a quadratic optimization problem in networks with $10$ agents and different graphs. For each graph with Laplacian matrix $\mathcal{L}$, we define the consensus matrix $W=I-\frac{1}{1+d_{max}}\mathcal{L}$ with $d_{max}$ being the largest degree of agents. The spectral gap of a graph is the difference between the two largest eigenvalues of its consensus matrix and reflects the connectivity of the agents. We simulate FlexPD-C algorithm and use its theoretical bounds for stepsize. The objective function at each agent $i$ is of the form $f_i(x)=c_i(x_i-b_i)^2$ with $c_i$ and $b_i$ being integers that are randomly chosen from $[1,10^3]$ and $[1,100]$. We run the simulation for $1000$ random seeds. On the $Y$-axis of part (b) of \cref{fig:radii}, we plot the average number of steps and communications until the relative error is less than $\epsilon=0.01$, i.e., $\frac{\norm{x^k-x^*}}{\norm{x^0-x^*}}<0.01$, and on the $X$-axis, from left to right, we have the spectral gaps of path, ring, 4-regular, random Erdos-Renyi (p=.9178), and complete graphs. As we observe in part (b) of \cref{fig:radii}, increasing the number of primal steps per iteration in poorly connected graphs improves the performance more significantly. Also, we notice that with respect to the number of communications 4-regular graph has the best performance.

\ff{\section{Concluding Remarks}\label{sec:conclusions}
In this paper, we propose a flexible framework of first-order primal-dual optimization algorithms for distributed optimization.
Our framework includes three classes of algorithms, which allow for multiple primal updates per iteration and are different in terms of computation and communication requirements. The design flexibility of the proposed framework can be used to control the trade-off between the execution complexity and the performance of the algorithm. We show that the proposed algorithms converge to the exact solution with a global linear rate. The use of this framework is not restricted to the distributed settings and it can be used to solve general equality constrained optimization problems satisfying certain assumptions. The numerical experiments show the convergence speed improvement of primal-dual algorithms with multiple primal updates per iteration compared to other known first-order methods like EXTRA and NEAR-DGD$^+$. Possible future work includes the extension of this framework to non-convex and asynchronous settings.}
\section*{Acknowledgments}

\bibliographystyle{plain}
\bibliography{CDCRef}
\end{document}